\documentclass[reqno,  a4paper]{amsart}
\usepackage{graphicx}
\usepackage{bm}
\usepackage{hyperref}
\usepackage{amssymb}
\usepackage{amsmath,amsthm, latexsym, wasysym}
\usepackage{enumerate}
\usepackage{amsfonts}
\usepackage{todonotes}

\usepackage{enumitem}
\newtheorem{theorem}{Theorem}[section]
\newtheorem{lemma}[theorem]{Lemma}

\newtheorem{proposition}[theorem]{Proposition}

\theoremstyle{definition}

\newtheorem{example}[theorem]{Example}

\numberwithin{equation}{section}
\usepackage[T1]{fontenc}
\usepackage{mathrsfs}
\usepackage{setspace}
\usepackage{graphicx,subfigure}
\usepackage{xcolor}
\usepackage{epstopdf}
\usepackage[utf8]{inputenc}
\usepackage{multicol}
\usepackage{amsmath}
\usepackage{amsthm}
\usepackage{upgreek}
\usepackage{amsfonts}
\usepackage{contour}
\usepackage{float}
\usepackage{breqn}
\usepackage{textcomp}
\usepackage{gensymb}
\usepackage[left=.9in,right=1in,top=1 in,bottom=1 in]{geometry}
\usepackage{hyperref}
\usepackage{lipsum}
\usepackage{apptools}
\newcommand{\remove}[1]{}
\theoremstyle{remark}

\newcommand{\ds} {\displaystyle}
\newcommand{\e}{\epsilon}

\newcommand{\al} {\alpha}

\newcommand{\de} {\delta}

\newcommand{\Om} {\Omega}
\newcommand{\ra} {\rightarrow}

\newcommand{\De} {\Delta}
\newcommand{\la} {\lambda}
\newcommand{\La} {\Lambda}
\newcommand{\noi} {\noindent}
\newcommand{\na} {\nabla}

\newcommand{\oline} {\overline}
\newcommand{\mb} {\mathbb}
\newcommand{\mc} {\mathcal}

\begin{document}
	
\title[Quasilinear Choquard problem with $N$-Laplacian]{Quasilinear Choquard equations involving $N$-Laplacian and critical exponential nonlinearity}
\author[Reshmi Biswas, Sarika Goyal and  K. Sreenadh]
{Reshmi Biswas, Sarika Goyal and K. Sreenadh}
\address{Reshmi Biswas \newline
	Department of Mathematics, IIT Delhi, Hauz Khas, New Delhi 110016, India}
\email{reshmi15.biswas@gmail.com}

\address{Sarika Goyal \newline
	Department of Mathematics, Bennett University, Greater Noida, Uttar Pradesh 201310, India}
\email{sarika.goyal@bennett.edu.in }

\address{K. Sreenadh \newline
	Department of Mathematics, IIT Delhi, Hauz Khas, New Delhi 110016, India}
\email{sreenadh@maths.iitd.ac.in}


	%\date{}
%	\author{ }\footnote{\thanks{The author was supported by ...........}  } \hspace{2mm}
	% {\bf\large Name}\vspace{1mm}\\
	
	%\date{\small{Received ---------------} \\
	%{\it\small Communicated by ----------}}
	
	\maketitle
	\begin{center}
		{\bf\small Abstract}
		
		\vspace{3mm} \hspace{.05in}\parbox{4.5in} {{\small In the present paper, we study  a class of quasilinear  Choquard equations involving $N$-Laplacian and the nonlinearity with the  critical exponential growth.	We discuss the existence of  positive solutions of such equations.} }
	\end{center}

%	\noindent
%	{\it \footnotesize 2020 Mathematics Subject Classification}. {\scriptsize 35A01, 35A15, 35A16, 35B09   }.\\
%	{\it \footnotesize Key words}. {\scriptsize Quasilinear equation, $N$-Laplacian, Choquard equation, Critical nonlinearity, Trudinger-Moser inequality, Variational methods.}
	
	%%%%%%%%%%%%%%%%%%%%%%%%%%%%%%%%%%%%%%%%%%%%%%%%%%%%%%%%%%%%%%%%%%%%%%%%%%%%%%%%%%%%%%%%%%%%%%%%%%%%%%%%%%%%%%%%%%%%%%%%%%%%%%%%%%%%%
	%
	%
	%
	\section{\bf Introduction}

	\noi This paper is concerned with the existence of the positive solutions for the following family of quasilinear equations with $N$-Laplacian and exponential Choquard type nonlinearity
	\begin{equation*}
	\label{pq}\tag{$P_{*}$}\left\{
	\begin{array}{l}
	-\Delta_N u -\Delta_N (u^{2})u=
	\left(\displaystyle\int_{\Om}\frac{F(y,u)}{|x-y|^{\mu}}~dy\right)f(x,u) \; \text{in}\;
	\Om,\\
	u> 0~\mbox{in}~ \Om,\\
	u=0 ~\mbox{on}~ \partial\Om,
	\end{array}
	\right.
	\end{equation*}
	where $\Om$ is a smooth bounded domain in $\mathbb R^N,\; N\geq2$, $0<\mu<N$ and $\Delta_N:=\text{div}\left(|\nabla u|^{N-2}\nabla u\right)$ is called the $N$-Laplacian.  
	\noi The function $f:\Om \times \mb R\to \mb R$ is given by $f(x,s)=g(x,s)\exp(|s|^{\frac{2N}{N-1}})$, where $g \in C(\bar \Om \times \mb R)$ satisfies some appropriate assumptions described later.  The function $F(y,s)$ is the primitive of $f(y,s)$ defined as $F(y,s)=\int_0^s f(y,t)dt.$\\
	
	\noi  The problems involving the quasilinear  operator $-\Delta_p u -\Delta_p (u^{2})u,\;1<p<\infty,$ has been of interest to many researchers for long due to its significant applications in the modeling of the physical phenomenon such as in plasma physics and fluid mechanics \cite{bass}, in dissipative quantum mechanics \cite{has}, etc.
	Solutions of such equation are related to the existence of standing wave solutions for quasilinear Schr\"odinger equations of
	the form
	\begin{align}\label{l}
	iu_t =-\Delta u+V(x)u-h_1(|u|^2)u-C\Delta h_2(|u|^2)h_2'(|u|^2)u,\;\; x\in\mathbb R^N,
	\end{align}
	where V is a  potential, $C$ is a real constant, $h_1$ and $h_2$ are real functions. Depending upon the different type of  $h_2$, the quasilinear equations of the form \eqref{l} appear in  the study of mathematical physics. For example,  in the modeling of the superfluid film equation in plasma physics, $h_2(s)=s$ (see \cite{1}) and for studying the self-channeling of a high-power ultra short laser in matter $h_2(s) =\sqrt{1+s^2}$ (see \cite{33}). \\
	
	\noi The main mathematical difficulty we face in studying the problem of  type \eqref{pq} occurs due to the
	quasilinear term $\Delta_N(u^2)u$, which doesn't allow the natural energy functional corresponding to the
	problem \eqref{pq} to be  well defined for all $u\in W_0^{1,N}(\Om)$ (defined in Section \ref{sec2}). Hence, we can not apply variational method directly for such problem. To overcome this 
	inconvenience, several methods and arguments have been developed, such as the perturbation method (see for e.g.,\cite{19,24}) a constrained minimization
	technique (see for e.g., \cite{20,22,28,29}),  and a
	change of variables (see for e.g., \cite {CJ,do,8,M-do1,11,12}).
	%
	% In [26], by a change of variables the quasilinear problem was transformed to a semilinear
	%one and an Orlicz space framework was used as the working space, and they were able to prove the existence
	%of positive solutions of (1.3) by the mountain pass theorem. The same method of changes of variables was
	%used in [10], but the usual Sobolev space $H^1(\mb R^N)$ framework was used as the working space and they studied
	%a different class of nonlinearity. In [27], the existence of both one sign and nodal ground state type solutions
	%was established by the Nehari method. Yang and Ding [45] showed that the existence of semiclassical states
	%for perturbed version of Schr\"odinger equation in $\mb R^N$ by using minimax methods.\\
	 The nonlinearity in the problem \eqref{pq}, which is nonlocal in nature, is driven by the Hardy-Littlewood-Sobolev inequality and the Trudinger-Moser inequality. Let us first recall the following well known Hardy-Littlewood-Sobolev inequality [Theorem 4.3, p.106] \cite{lieb}.
	\begin{proposition}\label{HLS}
		(\textbf {Hardy-Littlewood-Sobolev inequality}) Let $t$, $r>1$ and $0<\mu<N $ with $1/t+\mu/N+1/r=2$, $g_1 \in L^t(\mathbb R^n)$ and $g_2 \in L^r(\mathbb R^n)$. Then there exists a sharp constant $C(t,N,\mu,r)$, independent of $g_1,$ $g_2$ such that
		\begin{equation}\label{HLSineq}
		\int_{\mb R^N}\int_{\mb R^N} \frac{g_1(x)g_2(y)}{|x-y|^{\mu}}\mathrm{d}x\mathrm{d}y \leq C(t,N,\mu,r)\|g_1\|_{L^t(\mb R^N)}\|g_2\|_{L^r(\mb R^N)}.
		\end{equation}
		{ If $t =r = \textstyle\frac{2N}{2N-\mu}$ then
			\[C(t,N,\mu,r)= C(N,\mu)= \pi^{\frac{\mu}{2}} \frac{\Gamma\left(\frac{N}{2}-\frac{\mu}{2}\right)}{\Gamma\left(N-\frac{\mu}{2}\right)} \left\{ \frac{\Gamma\left(\frac{N}{2}\right)}{\Gamma(N)} \right\}^{-1+\frac{\mu}{N}}.  \]
			In this case there is equality in \eqref{HLSineq} if and only if $g_1\equiv (constant)g_2$ and
			\[g_1(x)= c_0(a^2+ |x-b|^2)^{\frac{-(2N-\mu)}{2}}\]
			for some $c_0 \in \mathbb C$, $0 \neq a \in \mathbb R$ and $b \in \mathbb R^N$.}
	\end{proposition}
	
	% The study of Choquard equations originates from the work of S. Pekar \cite{pekar} and P. Choquard \cite{choquard} where they used elliptic equations with Hardy-Littlewood-Sobolev type nonlinearity to describe the quantum theory of a polaron at rest and to model an electron trapped in its own hole in the Hartree-Fock theory, respectively. For more details on the application of Choquard equations, we refer to \cite{survey}.\\
	
	\noi Now a days,  an ample amount of attention has been attributed to the study of Choquard type equations, which was started 
	with the seminal work of S. Pekar \cite{pekar}, where the author considered 
	the following nonlinear Schr\"{o}dinger-Newton equation:
	\begin{align}\label{sn}
		-\Delta u + V(x)u = ({\mathcal{K}}_\mu * u^2)u +\la f_1(x, u)\;\;\; \text{in}\; \mathbb R^N,
		\end{align}
	where $\la>0,$ $\mathcal{K}_\mu$ denotes  the  Riesz potential, $V:\mb R^N\to \mb R$ is a continuous function and $f_1:\mb R^N\times\mb R\to \mb R$ is a Carath\'eodory function with some appropriate growth assumptions.
	When $\la=0$, the nonlinearity in the right-hand side of \eqref{sn} is termed as  Choquard type nonlinearty. In the application point of view, such type of nonlinearity plays an important role in  describing  the Bose-Einstein
	condensation (see \cite{bose}) and also, appears in the modeling of the self gravitational
	collapse of a quantum mechanical wave function (see \cite{penrose}).  P. Choquard (see \cite{choq}) studied such elliptic equations of type \eqref{sn} for construing the quantum theory of a polaron at rest and for modeling the phenomenon of an electron being trapped in its own hole in the Hartree-Fock theory.
	When $V(x)=1,\la=0$, the equations of type \eqref{sn} were   studied rigorously in  \cite {lieb,Lions}. 
	For more extensive study  of Choquard equations, without attempting to provide a complete list, 
	we refer to \cite{moroz,moroz-main,moroz4,moroz5} and the references therein.\\
	
	\noi The main feature of the problem \eqref{pq} is that the 
	nonlinear term $f(x,t)$ has the maximal growth on $t$, that is, critical exponential growth with respect
	to the following Trudinger–Moser inequality (see \cite{Trud-Moser}):
	\begin{theorem}\label{TM-ineq}(\textbf {Trudinger–Moser inequality})
		For $N\geq 2$, $u \in W^{1,N}_0(\Om)$
		\[\sup_{\|u\|\leq 1}\int_\Om \exp(\alpha|u|^{\frac{N}{N-1}})~dx < \infty\]
		if and only if $\alpha \leq \alpha_N$, where $\alpha_N = N\omega_{N-1}^{\frac{1}{N-1}}$ and $\omega_{N-1}=$ $(N-1)$- dimensional surface area of $\mb S^{N-1}$.
	\end{theorem}
\noi  Here the Sobolev space $W^{1,N}_0(\Om)$ and the corresponding norm $\|\cdot\|$
	are defined in Section \ref{sec2}. The critical growth non-compact problems associated to this inequality are initially studied by the work of  Adimurthi \cite{adi} and de Figueiredo et al. \cite{DR}. More recently, authors in \cite{GPS, GS1, GS2} studied the existence of  multiple positive solutions for quasilinear equations involving exponential nonlinearities. 
	%We say that $f(x,t)$ has critical growth at $+\infty$ $\al_0>0$ such that \begin{equation*}
	% \lim_{t\to\infty}\frac{f(t)}{\exp(\al |t|^{\frac{2N}{N-1}})}=\left\{
	% \begin{array}{l}
	% 0, \;\;
	% \forall \al>\al_0\\
	% +\infty,~ \forall\al<\al_0
	% \end{array}
	% \right.
	% \end{equation*}
	 In the case $1<p<N$, the nonlinearity is of polynomial growth and the critical growth is $ t^{2p^*},$ where  $p^*=Np/(N-p)$ (see \cite{deng,8,27}).
	% On the other hand, the study of elliptic equations involving nonlinearity with exponential growth are motivated by the following Trudinger-Moser inequality in \cite{Trud-Moser}:
	% \begin{theorem}\label{TM-ineq}
	% 	For $N\geq 2$, $u \in W^{1,N}_0(\Om)$
	% 	\[\sup_{\|u\|\leq 1}\int_\Om \exp(\alpha|u|^{\frac{N}{N-1}})~dx < \infty\]
	% 	if and only if $\alpha \leq \alpha_N$, where $\alpha_N = N\omega_{N-1}^{\frac{1}{N-1}}$ and $\omega_{N-1}=$ $(N-1)$- dimensional surface area of $\mb S^{N-1}$.
	% \end{theorem} 
	When $p=2$, the problem of type \eqref{pq} without the convolution term, that is the equation $$-\Delta u -\Delta (u^{2})u + V(x) u=f(x, u) \,\; \text{in}\;\; \mb R^2,$$ where $V:\mb R^N\to \mb R$ is a continuous potential and $f:\mb R^2\times\mb R\to \mb R$ is  a continuous function with some suitable assumption and is having critical exponential growth ($\exp(\al s^{4})$), was studied by Marcus do Ó et al. in \cite{do}. Later, Wang et al. \cite{wang} extended these results to the case of $p=N\geq2.$ \\
	
		%
	%Similar problem, for $p=2$ and critical exponential nonlinearities was studied in \cite{do}.  
	
	%But, it is important to point out that, in all the aforementioned works, the authors studied existence of only one solution, not the multiplicity. For multiplicity of solutions involving the quasilinear operator and polynomial type nonlinearity, when $p=2$, we refer to some very recent works of \cite{bibid} for the critical case, and \cite{bibid} for the sub-critical case.\\
	
	\noi Involving critical exponential Choquard type nonlinearity and Laplacian operator, Alves et al. \cite{yang-JDE} studied  the following problem 
	\[-\e^2\Delta u+V(x)u=\left(\displaystyle\int_{\mb R^2}\frac{F(y,u)}{|x-y|^{\mu}}~dy\right)f(x,u),\;\;\; x\in\mb R^2,\] where  $\e>0$, $0<\mu<2$, $V:\mb R^2\to\mb R$ is the continuous potential function with some particular properties and the continuous function $f:\mb R^2\times\mb R\to\mb R$ enjoys the  critical exponential growth with some appropriate assumptions. Consequently, for the higher dimension, that is for $N\geq2$, authors in \cite{AGMS} discussed the  Kirchhoff- Choquard problems involving the $N$-Laplacian with critical exponential growth.\\
	%$$-m\left (\int_\Om|\nabla u|^N dx\right )\Delta_N u=\left(\displaystyle\int_{\Om}\frac{F(y,u)}{|x-y|^{\mu}}~dy\right)f(x,u)\text {\;\;in\;\;} \Om,\;\;\; u>0 \text  {\;\;in\;\;} \Om,\;\;\; u=0 \text{\;\;on\;\;} \mb R^N\setminus \Om,$$ where $0<\mu<N$, $\Om\subset \mb R^N$ is a bounded domain and $f:\Om\times\mb R\to R$ and $m:\mb R^+\to\mb R$ are continuous functions with some appropriate assumption. 
	
	\noi Inspired from all the above mentioned works, in this article, we investigate the existence results for the problem \eqref{pq}, involving the Choquard type  critical nonlinearity motivated by the above inequality  \eqref{HLSineq}. The main contribution in this work is to identify the first critical level for this problem and study the Palais-Smale sequences below this level.  Unlike as in the case of critical exponential problem involving $N$-Laplacian, where we generally consider the critical exponential growth as $\exp(|t|^{N/(N-1)})$, in our problem \eqref{pq}, due to the presence of the quasilinear term $\De_N(u^2)u$ in the principal operator, the critical  growth in the critical dimension is $\exp(|t|^{2{N/(N-1)}})$. It turns out that the  first critical level  reduces by half (see Lemma \ref{PS-level}). In this sense, the analysis presented here is new even for the equations without the presence of the nonlocal term ($\int_\Om F(y,u) |x-y|^{-\mu} dy$) . Once we identify the first critical level, we prove some technically involved estimates to show the existence of non-trivial weak limit for Palais-Smale sequences at the mountain-pass min-max level. 
	We would like to mention that to the best of our knowledge, there is no work on the existence of positive solutions to the elliptic equations involving quasilinear operator  and critical exponential Choquard type nonlinearity. In this article, we study such equations for the first time. Our result is new even for the case $N=2$.\\
	%such that $f(x,t)=g(x,t) \exp({|t|^{4}})$, where $g(x,t)$ satisfies the following assumptions:
	%\begin{enumerate}
	%	\item[$(f1)$] $g\in C^1(\overline{\Om}\times \mb R)$, $g(x,t)=0,$ for all $t\le 0$, $g(x,t)>0,$ \text{for all}
	%	$t>0$.
	%	\item[$(f2)$] For any $\e>0,$ $\ds \lim_{t\ra \infty}\sup_{x\in \overline{\Om}} g(x,t) \exp(-\e |t|^{4} )=0$, $\ds\lim_{t\ra \infty}\inf_{x\in \overline{\Om}} g(x,t) \exp(\e|t|^{4})=\infty.$
	%	\item[$(f3)$] There exist positive constants $T$, $T_0$ and $\gamma_0$ such that
	%	\[ 0<t^{\gamma_0} F(x,t)\le T_0 f(x,t)\;\mbox{for all}\; (x,t)\in \Om\times[t_0,+\infty).\]
	%%	%\item[$(g4)$] For each $\ds x\in \Omega, \frac{g(x,t)}{t^{2\frac{n}{s}-1}}$ is increasing for $t>0$ and $\ds\lim_{t\rightarrow 0^+} \frac{g(x,t)}{t^{2\frac{n}{s}-1}}=0,\;\text{uniformly in }\; x\in \Om.$
	%%	%\item[$(g5)$] $\ds \lim_{t\rightarrow \infty} t h(x,t)=\infty.$
	%%	\item[$(f4)$] For $\gamma >1$ (defined in (M2)), there exists a $l>\frac{\gamma n}{2s}-1$ such that the map $t \mapsto \frac{g(x,t)}{t^{l}}$ is increasing on $\mb R^{+} \setminus \{0\}$, uniformly in $x\in \Om$.
	%\end{enumerate}
	%%%%%%%%%%%%%%%%%%%%%%%%%%%%%%%%%%%%%%%%%%%%%%%%%%%%%%%%%%%%%%%%%%%%%%%%%%%%%%%%%%%%%%%%%%%%%5
	
	%$V:\mb \mb R\rightarrow \mb R$ is {a continuous function} such that
	%\begin{enumerate}
	%\item [$(V_1)$] There exists $V_0>0$ such that $V(x)>V_0>0$ for all $x\in \mb R^2.$
	%\item [$(V_2)$] $V(x)\leq \lim_{|x|\to\infty} V(x)=V_\infty<\infty$ with $V(x)\not= V_\infty.$
	%\end{enumerate}
	%
	\noi We now state all  the hypotheses imposed  on the continuous function $f:\Om\times\mb R\to\mb R$, given by $f(x,s)= g(x,s)\exp\left(s^{\frac{2N}{N-1}}\right)$:
	\begin{enumerate}
		\item[$(f_1)$] $g\in C(\overline{\Om}\times \mb R)$ such that for each $x\in\overline\Om$,  $g(x,s)=0,$ for all $s\le 0$ and $g(x,s)>0,$ \text{for all}
		$s>0$.
		%	\item[$(f_2)$] There exists  $\al_0>0$ such that \begin{equation*}
		%	\lim_{t\to\infty}\frac{f(t)}{\exp(\al |t|^{\frac{2N}{N-1}})}=\left\{
		%	\begin{array}{l}
		%	0, \;\;
		%	\forall \al>\al_0\\
		%	+\infty,~ \forall\al<\al_0
		%	\end{array}
		%	\right.
		%	\end{equation*}
		%	\item[$(f_2)$] There exists {\color{red} $l>N-1$} such that $\displaystyle\frac{f(x,s)}{t^{l}}$ is increasing in $s\in\mb R^N\setminus\{0\}$, uniformly $x\in\Om$.
		\item[$(f_2)$]  $\displaystyle\lim_{s\to0}\frac{f(x,s)}{s^{N-1}}=0$, uniformly in $x\in\Om$.
		\item[$(f_3)$] For any $\e>0$, $\displaystyle\lim_{s\to\infty}\sup_{x\in\overline\Om}g(x,s)\exp\left(-\e|s|^{\frac{2N}{N-1}}\right)=0$ and $\displaystyle\lim_{s\to\infty}\inf_{x\in\overline\Om}g(x,s) {\exp(\e |s|^{\frac{2N}{N-1}})}=\infty$.
		\item[$(f_4)$] There exist positive constants  $s_0$ and $m_0$ such that
		\[ 0<s^{m_0} F(x,s)\le M_0 f(x,s)\;\mbox{for all}\; (x,s)\in \mb R^2\times[s_0,+\infty).\]
		\item[$(f_5)$]  There exists $\tau>N$  such that $0<\tau F(x,s)\leq f(x,s)s,$ for all $ (x,s)\in \Om\times (0,\infty).$
		\item[$(f_6)$] We assume that\begin{equation}\label{h-growth}
		\displaystyle \lim_{s\to +\infty} \frac{sf(x,s)F(x,s)}{\exp\left(2 |s|^{\frac{2N}{N-1}}\right)} = \infty,\mbox{ uniformly in }x \in \overline{\Om}.
		\end{equation}
	\end{enumerate}
	%\begin{remark}
	%From the condition $(f_2)$ we can deduce the following:
	%\begin{itemize}
	%\item[$(f_6)$]  $\frac{f(x,s)}{s^{N-1}}$ is increasing in $s>0$ and $\displaystyle\lim_{s\to0}\frac{f(x,s)}{s^{N-1}}=0$, uniformly in $x\in\Om$.
	%\item[$(f_7)$] There exists $\tau>N$ such that $0<\tau F(x,s)\leq f(x,s)s,$ for all $ s>0.$
	%\end{itemize}
	%\end{remark}
	\begin{example}
		Consider $f(x,s)= g(x,s)e^{|s|^{\frac{2N}{N-1}}}$, where $g(x,s)=\left\{\begin{array}{lr}
		{t^{a_0+ \left({N}-1\right)} \exp(d_0s^{r})},\; \mbox{if} \; s> 0\\
		0, \; \mbox{if} \; s\leq 0	\end{array}\right.$
		for some $a_0>0$, $0<d_0\leq \alpha_{N}$ and $1\leq r <\frac{2N}{N-1}$. Then $f$ satisfies all the conditions from $(f_1)$-$(f_6)$.
	\end{example}
	Our  main result reads as following:
	
	\begin{theorem}\label{T1}
		Let $\Om\subset\mathbb R^N\; (N\geq2)$ be a smooth bounded domain and let the hypotheses $(f_1)$-$(f_6)$ hold. Then the problem \eqref{pq} has a non trivial positive weak solution.
	\end{theorem}

	\textbf{Notation.}  Throughout this paper, we make use of the following notations:
	\begin{itemize}
		\item If $u$ is a measurable function, we denote  the positive  and negative parts by $u^{+}=\max\left\{u,0\right\}$ and $u^{-}=\max\left\{-u,0\right\}$, respectively. 
		\item If $A$ is a measurable set in $\mathbb{R}^{N}$, we the Lebesgue measure of $A$  by $\vert A \vert$ . 
		\item The arrows $\rightharpoonup $ , $\to $ denote weak convergence,  strong convergence, respectively.
		\item $B_r(x)$ denotes the ball of radius $r>0$ centered at $x\in \Om.$
		\item  $\overline{B_r(x)}$ denotes the closure of the set $B_r(x)$ with respect to $W_0^{1,N}(\Om)$-norm topology.
		\item  $\overline{\mc S}$ denotes the closure of a set $\mc S\subset \mb R^N$.
%		\item For any function $\tilde f$, $\tilde f=o(l)$ as $l\to\infty$ implies $\lim_{l\to\infty}\tilde f= 0.$ 
\item $A\subset\subset B$ implies $\oline A$ is compact in $B$.
	\item $c,C_{1},C_{2},\cdots, \tilde C_1, \tilde C_2,\cdots, C$ and $\tilde C$ denote (possibly different from line to line) positive constants.
	\end{itemize}
	
	\section{Preliminaries and Variational Set-up}\label{sec2}
	For $u:\Om\to\mathbb R$, measurable function, and  for $1\leq p \leq \infty$, we define the Lebesgue space $L^{p}(\Om)$ as 
	$$L^{p}(\Om)=\{u:\Om\to\mathbb R\;\; \text{measurable}| \int_\Om| u|^p dx<\infty\}$$ equipped with the usual norm denoted by $\Vert u\Vert_{L^p(\Om)}$.
	Now the Sobolev space $W^{1,p}_0(\Om)$ is defined as 
	$$W^{1,p}_0(\Om)=\{u\in L^p(\Om)| \int_\Om|\nabla u|^p dx<\infty\}$$ which is  endowed with  the norm
	%$$\|u\|_{W_0^{1,p}(\Om)}=\|u\|_{L^p(\Om)}+\displaystyle \left({\int_{\Om} |\nabla u|^p dx}\right)^{1/p}.$$  This is equivalent to the following Banach norm:
	$$\|u\|:=\displaystyle \left({\int_{\Om} |\nabla u|^p dx}\right)^{1/p}.$$
	%\noi To study our problem involving critical exponential growth, first we recall the following Trudinger-Moser inequality ( see \cite{Trud-Moser}):
	%\begin{theorem}\label{TM-ineq}
	%	For $N\geq 2$, $u \in W^{1,N}_0(\Om)$
	%	\[\sup_{\|u\|\leq 1}\int_\Om \exp(\alpha|u|^{\frac{N}{N-1}})~dx < \infty\]
	%	if and only if $\alpha \leq \alpha_N$, where $\alpha_N = N\omega_{N-1}^{\frac{1}{N-1}}$ and $\omega_{N-1}=$ $(N-1)$- dimensional surface area of $\mb S^{N-1}$.
	%\end{theorem} 
We have that the embedding $W^{1,N}_0(\Om)\ni u \mapsto \exp(|u|^\beta)\in L^1(\Om)$ is compact for all $\beta \in \left[1, \frac{N}{N-1}\right)$ and is continuous for $\beta = \frac{N}{N-1}$. Consequently, the map $\mc M: W^{1,N}_0(\Om) \to L^q(\Om)$, for $q \in [1,\infty)$, defined by $\mc M(u):= \exp\left( |u|^{\frac{N}{N-1}}\right)$ is continuous with respect to the norm topology.
	%We define the following subset of $ H^{1}(\mb R^2):$
	%$$X=\left\{u\in H^{1}(\mb R^2) | \int_{\mb R^2}V(x)|u|^2dx <\infty,\,\int_{\mb R^2}u^2|\nabla u|^2dx <\infty\right\}$$
	
	% \begin{definition}\label{weak}
	%We say that a function $u \in W^{1,N}_0(\Om)$ is a weak solution (solution, for short) of \eqref{pq} if for every $\psi \in C_c^\infty(\Om)$,
	%\begin{align}\label{ws}
	%\displaystyle\int_{\Om} (1+u^{2})\nabla u \nabla \psi dx+\int_{\Om} u|\nabla u|^2\psi dx-\frac 12\int_{\Om}\int_{\Om} \frac{F(y,u(y))f(x,u(x))}{|x-y|^\mu}\psi(x)dxdy=0.
	%\end{align}
	% \end{definition}
\noi	The  natural energy functional  associated to the problem \eqref{pq} is the following:
	%$I(u):W^{1,N}_0(\Om)\to\mb R$, 
	\begin{align*}I(u)=\frac{1}{N}\displaystyle\int_{\Om }(1+2^{N-1}|u|^{N})|\nabla u|^{N}dx-\frac 12\int_{\Om}\int_{\Om} \frac{F(y,u(y))F(x,u(x))}{|x-y|^\mu}dxdy.\end{align*}
	
\noi	Observe that, the functional $I$ is not well defined in $W^{1,N}_0(\Om)$ due to the fact that $\displaystyle\int_{\Om} u^{N}|\nabla u|^{N}dx$ is not finite for all $u\in W_0^{1,N}(\Om) $. So, it is difficult to apply variational methods directly in our problem \eqref{pq}. In order to overcome this difficulty, we apply the following change of variables which was introduced in \cite{CJ}, namely, $w:=h^{-1}(u),$ where $h$ is defined by
	\begin{equation}\label{g}
	\left\{
	\begin{array}{l}
	h^{\prime}(s)=\displaystyle\frac{1}{\left(1+2^{N-1}|h(s)|^{N}\right)^{\frac{1}{N}}}~~\mbox{in}~~ [0,\infty),\\
	h(s)=-h(-s)~~\mbox{in}~~ (-\infty,0].
	\end{array}
	\right.
	\end{equation}
	
\noi	Now we gather some properties of $h$, which we follow throughout in this article. For the detailed proofs of such results, one can see  \cite{CJ,do}.
	\begin{lemma}\label{L1} The function $h$ satisfies the following properties:
		\begin{itemize}
			\item[$(h_1)$] $h$ is uniquely defined, $C^{\infty}$ and invertible;
			\item[$(h_2)$] $h(0)=0$;
			\item[$(h_3)$] $0<h^{\prime}(s)\leq 1$ for all $s\in \mathbb{R}$;
			\item[$(h_4)$] $\frac{1}{2}h(s)\leq sh^{\prime}(s)\leq h(s)$ for all $s>0$;
			\item[$(h_5)$] $|h(s)|\leq |s|$ for all $s\in \mathbb{R}$;
			\item[$(h_6)$] $|h(s)|\leq 2^{1/{(2N)}}|s|^{1/2}$ for all $s\in \mathbb{R}$;
			\item[$(h_7)$] $\ds \lim_{s\to+\infty}{h(s)}/{s^{\frac 12}}=2^{\frac{1}{2N}}$;
			\item[$(h_8)$]  $|h(s)|\geq h(1)|s|$ for $|s|\leq 1$ and $|h(s)|\geq h(1)|s|^{1/2}$ for $|s|\geq 1$;
			\item[$(h_9)$] $h^{\prime \prime}(s)<0$ when $s>0$ and $h^{\prime \prime}(s)>0$ when $s<0$.
			%		\item[$(h_{10})$] the function $(h(s))^{1-\gamma}$ for $\gamma>1$ is decreasing for all $s>0$;
			%		\item[$(h_{11})$] the function  $(h(s))^{-\gamma}$ is decreasing for all $s>0$;
			%		\item[$(h_{12})$] $|h(s)h^{\prime}(s)|<1/ \sqrt[]{2}$ for all $s\in \mathbb{R}$;
			%		\item[$(h_{13})$] $h^{2}(st)\geq th^{2}(s)$ for all $t\geq 1$ and $s\geq 0$.
		\end{itemize}
		
	\end{lemma}
	\begin{example}
		One of the example of such functions is given in 
		\begin{figure}[H]
			%		\centering
			\includegraphics[width=80mm]{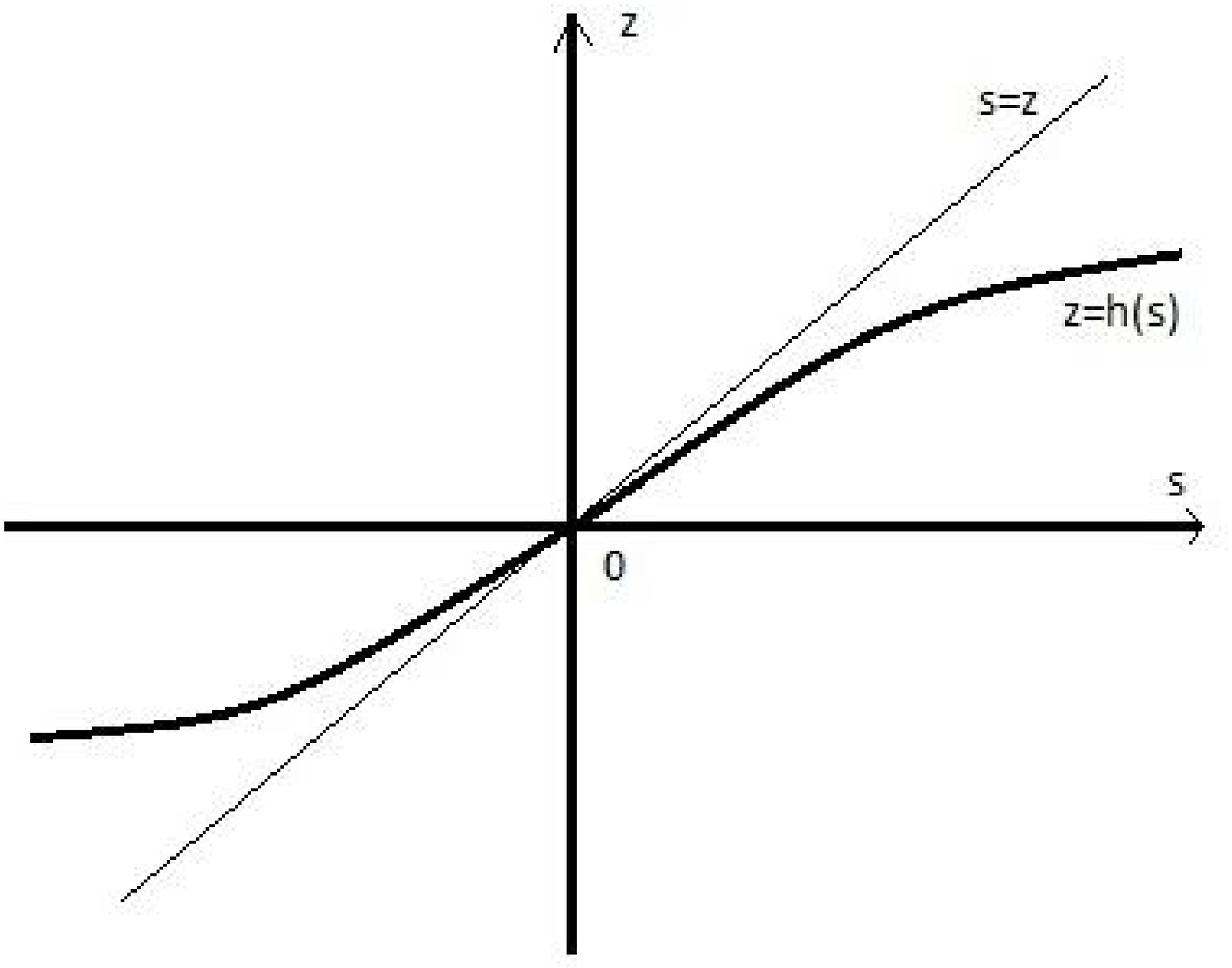}
			\caption{Plot of the function $h$}
			%\label{pic}
		\end{figure}	%Figure \ref{pic}.
	\end{example}
\noi	After employing the change of variable  $w=h^{-1}(u),$ we define the new functional $J:W^{1,N}_0(\Om)\to \mb R$ as
	\begin{align}\label{energy}J(w)=\frac{1}{N}\displaystyle\int_{\Om} |\nabla w|^{N}dx-\frac 12\int_{\Om}\int_{\Om} \frac{F(y,h(w))F(x,h(w))}{|x-y|^\mu}dxdy.\end{align}
	From the properties of $f,h$, it can be derived that the functional $J$ is well defined  and $J\in C^1.$
	% We, define the following subset of $ H^{1}(\mb R^2):$
	% $$E:=\left\{u\in H^{1}(\mb R^2) | \int_{\mb R^2}|\nabla u|^2 dx+\int_{\mb R^2}V(x)|u|^2dx <\infty,\right\}$$ equipped with the following norm 
	% $$\|u\|=  \left(\int_{\mb R^2}|\nabla u|^2 dx+\int_{\mb R^2}V(x)|u|^2dx\right)^{1/2}.$$
	% Observe that $\|\cdot\|$ and $\|\cdot\|_{H^1(\mb R^2)}$ are equivalent in $E$.
	We observe that if  $w \in W^{1,N}_0(\Om)$ is a critical point of the functional $J$, that is for every $v \in W^{1,N}_0(\Om)$,
	\begin{align*}
	\int_{\Om} |\nabla w|^{N-2}\nabla w\nabla v dx-\int_{\Om}\int_{\Om} \frac{F(y,h(w)(y))f(x,h(w)(x))}{|x-y|^\mu}h'(w)v(x)dxdy=0,\end{align*} then $w$ is a weak solution (solution, for short)  to the following problem:
	\begin{equation}
	\label{pp}\left\{
	\begin{array}{l}
	-\Delta_N w=
	\left(\displaystyle\int_{\mb R^2}\frac{F(y,h(w))}{|x-y|^{\mu}}~dy\right)f(x,h(w))h'(w) \; \text{in}\;
	\Om,\\
	w> 0~\mbox{in}~ \Om,\\
	w=0 ~\mbox{on}~ \partial\Om.
	\end{array}
	\right.
	\end{equation}
	It is easy to see that problem \eqref{pp} is equivalent to our problem \eqref{pq}, which takes $u = h(w)$ as its solutions. Thus, our main objective is now reduced to proving the existence of the solution of  \eqref{pp}.
	
	\section{Proof of the main theorem}
	\subsection{Mountain pass geometry}

	We begin this section with the study of mountain pass structure and Palais-Smale sequences corresponding to the energy functional  $J: W^{1,N}_0(\Om)\rightarrow \mb R$ associated to \eqref{pp}.
	From the assumptions, $(f_2)$-$(f_3)$, we obtain that for any $\e>0$, $r\geq 1$,  there exist $\tilde C(N,\e)$ and $C(N,\e)>0$ such that
	\begin{align} 
		&|f(x,s)| \le \e |s|^{N-1 }+ \tilde C(N,\e) |s|^{r-1} \exp\left((1+\e)|s|^{\frac{2N}{N-1}}\right),\;\; \text{for all}\; (x,s)\in \Om \times \mb R,\label{f}\\
	&|F(x,s)| \le \e |s|^N + C(N,\e) |s|^r \exp\left((1+\e)|s|^{\frac{2N}{N-1}}\right),\;\; \text{for all}\; (x,s)\in \Om \times \mb R.\label{k1}
	\end{align}
	Thus, for any $u \in W^{1,N}_0(\Om)$, in light of Sobolev embedding, we have $u \in L^q(\Om)$ for all $q \in [1,\infty)$, %Since $\mu \in (0,n)$, using H\"{o}lder's inequality, we get that
	%\[\int_\Om |u|^{\frac{2pn}{2n-\mu}}\exp\left(\frac{2pn}{2n-\mu}|u|^{\frac{n}{n-1}}\right)~dx \leq \left(\int_\Om |u|^{\frac{2n\alpha^\prime}{2n-\mu}} \right)^{1/\alpha^\prime} \left(\exp\left(\frac{2n\alpha}{2n-\mu}|u|^{\frac{n}{n-1}}\right) \right)^{1/\alpha}< +\infty\]
	%where $\alpha> 1$.
	which implies that
	\begin{equation}\label{KC-new1}
	F(x,u) \in L^{q}(\Om)\mbox{ for any }q\geq 1.
	\end{equation} Now by \eqref{g}, we have if $w\in W^{1,N}_0(\Om),$ then $h(w)\in W^{1,N}_0(\Om)$, which together with
	Proposition \ref{HLS} implies that
	\begin{align}\label{k2}
	\int_{\Om}\left(\int_{\Om}\frac{F(y,h(w))}{|x-y|^{\mu}}dy\right)F(x,h(w))~dx  \leq C(N,\mu){\|F(\cdot,h(w))\|^2_{L^{\frac{2N}{2N-\mu}}(\Om)}}.
	\end{align}
	%%%%%%%%%%%%%%%%%%%%%%%%%%%%%%%%%%%%%%%%%%%%%%%%%%%%%%%%%%%%%%%%%%%%%%%%%%%%%%%%%5
	%\begin{lemma}\label{lem7.1}
	%	Assume that the conditions $(f_1)$-$(f_5)$ hold. Then  we have the following:
	%	\begin{enumerate}
	%		\item [$(i)$] There exist $\rho_*>0$ and $R_*>0$ such that
	%		\begin{align*}
	%		J(w)\geq R_*>0 \quad \text{for any } w\in W^{1,N}_0(\Om) \text{ with }\ \|w\|=\rho_*.
	%		\end{align*}
	%		\item [$(ii)$] There exists $v_* \in W^{1,N}_0(\Om) $ such that $ J(v_*)<0$.
	%	\end{enumerate}
	%\end{lemma}
	\begin{lemma}\label{lem7.1}
		Assume that the conditions $(f_1)$-$(f_5)$ hold and let $h$ be defined as in \eqref{g}. Then  there exist $\rho_*>0$ and $R_*>0$ such that
		\begin{align*}
		J(w)\geq R_*>0 \quad \text{for any } w\in W^{1,N}_0(\Om) \text{ with }\ \|w\|=\rho_*.
		\end{align*}
	\end{lemma}
	\begin{proof} Let $w\in W^{1,N}_0(\Om$.	From  \eqref{k2}, \eqref{k1}, H\"{o}lder inequality, Sobolev inequality and Lemma \ref{L1}-$(h_5),(h_6)$, we have
			\begin{align}\label{kc-MP1}
			&\int_{\Om}\left(\int_{\Om} \frac{F(y,h(w))}{|x-y|^{\mu}}dy\right)F(x,h(w))~dx \notag\\
			& \leq C(N,\mu)\|\e |h(w)|^N + C(\e) |h(w)|^r \exp\left((1+\e)|h(w)|^{\frac{2N}{N-1}}\right)\|_{\frac{2N}{2N-\mu}}^2\notag\\
			&\leq C(\mu)\left[2^{\frac{2N}{2N-\mu}}\left\{\e^{\frac{2N}{2N-\mu}} \int_{\Om} |h(w)|^{\frac{2N^2}{2N-\mu}} dx\right.\right.\notag\\ & + \left.\left. (C(\e))^{\frac{2N}{2N-\mu}} \int_{\Om} |h(w)|^{\frac{2Nr}{2N-\mu}}\exp\left(\frac{2N}{2N-\mu}(1+\e)|h(w)|^{\frac{2N}{N-1}}\right)  dx\right\} \right]^{\frac{2N-\mu}{N}} \notag\\
	&\leq C(N,\mu)\left[2^{\frac{2N}{2N-\mu}}\left\{\e^{\frac{2N}{2N-\mu}} \int_{\Om} |w|^{\frac{2N^2}{2N-\mu}} dx \right. \right. \notag\\ & \quad + \left.\left. (C(\e))^{\frac{2N}{2N-\mu}} \int_{\Om}2^{\frac{r}{2N-\mu}} |w|^{\frac{Nr}{2N-\mu}}\exp\left(\frac{2N}{2N-\mu}(1+\e) 2^{\frac{1}{N-1}}|w|^{\frac{N}{N-1}}\right)  dx\right\} \right]^{\frac{2N-\mu}{N}} \notag\\
	&\leq C(N,\mu)4 \left[{\e^{2} }  \|w\|^{2N}_{L^{\frac{2N^2}{2N-\mu}}(\Om)} + (C(\e))^2 2^{\frac rN} \|w\|^{2r}_{L^{\frac{2Nr }{2N-\mu}}(\Om)} \left\{\int_{\Om}\exp\left(\frac{4N(1+\e)}{2N-\mu}2^{\frac{1}{N-1}}\left({|w|}\right)^{\frac{N}{N-1}}\right)dx\right\}^{\frac{2N-\mu}{2N}} \right]\notag\\
			&\leq C_1(N,\mu,\e) \left[ \|w\|^{2N} + \|w\|^{2r} \left\{\int_{\Om}\exp\left(\frac{4N(1+\e)}{2N-\mu}2^{\frac{1}{N-1}}\left({|w|}\right)^{\frac{N}{N-1}}\right)dx\right\}^{\frac{2N-\mu}{2N}} \right]\notag\\
			&\leq C_1(N,\mu,\e) \left[ \|w\|^{2N} + \|w\|^{2r} \left\{\int_{\Om}\exp\left(\frac{4N(1+\e)}{2N-\mu}2^{\frac{1}{N-1}}\|w\|^{\frac{N}{N-1}}\left(\frac{|w|}{\|w\|}\right)^{\frac{N}{N-1}}\right)dx\right\}^{\frac{2N-\mu}{2N}} \right]
			%		\notag\\
			%		&\leq C \left[{\e^{\frac{2N}{2N-\mu}} } \int_{\Om} |w|^{\frac{2N^2}{2N-\mu}}dx + C_1(\e) \|w\|^{\frac{2Nr }{2N-\mu}} \left\{\int_{\Om}\exp\left(\frac{4N 2^{\frac{1}{N-1}}(1+\e)}{2N-\mu}\left(|w|^{\frac{N}{N-1}}\right)\right)dx\right\}^{\frac12} \right]^{\frac{2N-\mu}{N}}\notag\\
			%		&\leq C \left[{\e^{\frac{2N}{2N-\mu}} } \int_{\Om} |w|^{\frac{2N^2}{2N-\mu}} + C_1(\e) \|w\|^{\frac{2Nr }{2N-\mu}} \left\{\int_{\Om}\exp\left(\frac{4N2^{\frac{1}{N-1}}(1+\e)\|w\|^{\frac{N}{N-1}}}{2N-\mu}\left(\frac{|w|}{\|w\|}\right)^{\frac{N}{N-1}}\right)dx\right\}^{\frac12} \right]^{\frac{2N-\mu}{N}}\notag\\
			%		&\leq C \left[{\e^{2} } \|w\|^{2p}_{L^{\frac{2N^2}{2N-\mu}}(\Om)} + C_1(\e) \|w\|^{\frac{2Nr }{2N-\mu}} \left\{\int_{\Om}\exp\left(\frac{4N2^{\frac{1}{N-1}}(1+\e)\|w\|^{\frac{N}{N-1}}}{2N-\mu}\left(\frac{|w|}{\|w\|}\right)^{\frac{N}{N-1}}\right)dx\right\}^{\frac12} \right]^{\frac{2N-\mu}{N}}
			% &\leq C_1 \left(\e \|u\|^{2\alpha} + C_2(\e) \|u\|^{2r} \left( \int_\Om\exp\left(\frac{4n(1+\e)\|u\|^{\frac{n}{n-s}}}{2n-\mu}\left(\frac{|u|}{\|u\|}\right)^{\frac{n}{n-s}}\right)\right)^{^{\frac{2n-\mu}{2n}}} \right) .
			\end{align}
		\noi	So, for sufficiently small $\e>0$ if we choose $w$ such that $\|w\|$ is small enough so that $$\displaystyle\frac{4N2^{\frac{1}{N-1}}(1+\e)\|w\|^{\frac{N}{N-1}}}{2N-\mu} <\al_N,$$ then using Theorem \ref{TM-ineq}  in \eqref{kc-MP1}, we obtain
		\begin{align}\label{1}
		\int_{\Om}\left(\int_{\Om} \frac{F(y,h(w))}{|x-y|^{\mu}}dy\right)F(x,h(w))~dx 
		\leq C_1(N,\mu,\e) \left( \|w\|^{2N}  +  \|w\|^{2r} \right).
		\end{align}
		%Using the approach as in [\cite{fang}, {Lemma 3.3}], we can derive  that there exist $C_0>0$, $\rho_0>0$ such that
		%\begin{align}\label{2}
		%	\int _{\Om}(|\nabla w|^2+ V(x) h^2(w)) dx\geq C_0 \|w\|^2
		%	\end{align}
		%	whenever $\|w\|\leq\rho_0$.
		Using \eqref{energy}, \eqref{1}, we have
		\begin{align*}
		J(w) &\geq \frac1N \|w\|^{N}-   C_1(N,\mu,\e) \left( \|w\|^{2 N}  +  \|w\|^{2r}
		\right).
		\end{align*}
		Now by taking  $r>0$ such that  $2r>N$, we can choose $0<\rho_*<1$  sufficiently small so that, we finally obtain $J(w) \geq R_* >0$ for all $w\in W^{1,N}_0(\Om)$ with $\|w\|=\rho_*$ and for some $R_*>0$ depending on $\rho_*$.\end{proof}
	
	\begin{lemma}\label{lem7.1.2}
		Assume that the conditions $(f_1)$-$(f_5)$ hold and let $h$ be defined as in \eqref{g}. Then  there exists $v_* \in W^{1,N}_0(\Om) $ with $\|v_*\|>\rho_*$ such that $ J(v_*)<0$ , where $\rho_*$ is   given as in Lemma \ref{lem7.1}.
	\end{lemma}
	\begin{proof}
		The condition $(f_5)$ implies that there exist some positive constant $C_1, C_2>0$ such that
		\begin{equation}\label{new2}F(x,s) \geq C_1s^{\tau}-C_2\;\text{ for all}\; (x,s) \in\Om \times [0,\infty).
		\end{equation}
		%	Let $\phi \in C_c^\infty(\Om)$ such that $supp(\phi)=\overline {B_1(0)}$. Then for large $t>1$, we have 
		%	\begin{align}\label{h}
		%h(t\phi)\geq \phi(x) h(t).
		%	\end{align}  Indeed, from Lemma \ref{L1}-$(h_4)$, one can easily compute that $\frac{ h(t)}{t}$ is  decreasing for $t>0$,  which together with the fact that $t\phi(x)\leq t$ gives
		%	$$\frac{h(t\phi)}{t\phi}\geq\frac{h(t)}{t}.$$ This eventually implies \eqref{h}. Now by Lemma \ref{L1}-$(h_5)$,$(h_8)$, \eqref{new2} and \eqref{h}, for large $t>1$, we obtain
		%	\begin{align*}
		%	&\int_{\Om} \left(\int_{\Om} \frac{F(y, h(
		%		))}{|x-y|^{\mu}}dy\right)F(x, h(t \phi))~dx\\ &\geq \int_{\Om} \int_{\Om} \frac{(C_1 (h(t \phi))^{d+1}(y)-C_2)(C_1 (h(t \phi))^{d+1}(x)-C_2)}{|x-y|^{\mu}}~dxdy\\
		%	& = C_1^2  \int_{\Om} \int_{\Om} \frac{(h(t)\phi(y))^{d+1} (h(t)\phi(x))^{d+1}}{|x-y|^\mu}~dxdy \\
		%	& \quad -2C_1C_2\int_{\Om} \int_{\Om}\frac{(h(t)\phi(x))^{d+1}}{|x-y|^\mu}~dxdy + C_2^2 \int_{\Om} \int_{\Om} \frac{1}{|x-y|^{\mu}}~dxdy\\
		%	&\geq C_1^2 (h(t))^{2(d+1)} \int_{\Om} \int_{\Om} \frac{(\phi(y))^{d+1}(\phi(x))^{d+1}}{|x-y|^{\mu}}~dxdy\\
		%	& \quad -2C_1C_2 (h(t))^{d+1}\int_{\Om} \int_{\Om}\frac{\phi^{d+1}(y)}{|x-y|^\mu}~dxdy + C_2^2 \int_{\Om} \int_{\Om} \frac{1}{|x-y|^{\mu}}~dxdy\\
		%	&={t^N}\left[C_3 
		%	\frac{(h(t))^{2(d+1)}}{t^N}-C_4 \frac{(h(t))^{d+1}}{t^N}+\frac{C_5}{t^N}\right].
		%%	&\geq CC_1^2 t^{(d+1)} \int_{\Om} \int_{\Om} \frac{(\phi(y))^{d+1}(\phi(x))^{d+1}}{|x-y|^{\mu}}~dxdy\\
		%%	& \quad -2C_1C_2t^{d+1}\int_{\Om} \int_{\Om}\frac{\phi^{d+1}(y)}{|x-y|^\mu}~dxdy + C_2^2 \int_{\Om} \int_{\Om} \frac{1}{|x-y|^{\mu}}~dxdy.
		%	\end{align*}
		Let $\phi (\geq0) \in C_c^\infty(\Om)$ such that $\|\phi\|=1$. 
		%\begin{align}\label{h}
		%h(t\phi)\geq \phi(x) h(t).
		%\end{align}  Indeed, from Lemma \ref{L1}-$(h_4)$, one can easily compute that $\frac{ h(t)}{t}$ is  decreasing for $t>0$,  which together with the fact that $t\phi(x)\leq t$ gives
		%$$\frac{h(t\phi)}{t\phi}\geq\frac{h(t)}{t}.$$ This eventually implies \eqref{h}. 
		Now by Lemma \ref{L1}-$(h_6)$,$(h_8)$ and \eqref{new2}, for large $t>1$, we obtain
		\begin{align*}
		&\int_{\Om} \left(\int_{\Om} \frac{F(y, h(t\phi
			))}{|x-y|^{\mu}}dy\right)F(x, h(t \phi))~dx\\ &\geq \int_{\Om} \int_{\Om} \frac{(C_1 (h(t \phi(y)))^{\tau}-C_2)(C_1 (h(t \phi(x)))^{\tau}-C_2)}{|x-y|^{\mu}}~dxdy\\
		& = C_1^2  \int_{\Om} \int_{\Om} \frac{(h(t\phi(y)))^{\tau} (h(t\phi(x)))^{\tau}}{|x-y|^\mu}~dxdy \\
		& \quad -2C_1C_2\int_{\Om} \int_{\Om}\frac{(h(t\phi(x)))^{\tau}}{|x-y|^\mu}~dxdy + C_2^2 \int_{\Om} \int_{\Om} \frac{1}{|x-y|^{\mu}}~dxdy\\
		&\geq C_1^2 (h(1))^{2\tau} t^\tau \int_{\Om} \int_{\Om} \frac{(\phi(y))^{\frac \tau 2}(\phi(x))^{\frac \tau 2}}{|x-y|^{\mu}}~dxdy\\
		%&={t^N}\left[C_3 
		%\frac{(h(t))^{2(d+1)}}{t^N}-C_4 \frac{(h(t))^{d+1}}{t^N}+\frac{C_5}{t^N}\right].
		%	&\geq CC_1^2 t^{(d+1)} \int_{\Om} \int_{\Om} \frac{(\phi(y))^{d+1}(\phi(x))^{d+1}}{|x-y|^{\mu}}~dxdy\\
		%	& \quad -2C_1C_2t^{d+1}\int_{\Om} \int_{\Om}\frac{\phi^{d+1}(y)}{|x-y|^\mu}~dxdy + C_2^2 \int_{\Om} \int_{\Om} \frac{1}{|x-y|^{\mu}}~dxdy.
		& \quad -2C_1C_2 2^{\frac{\tau}{2N}} t^{\frac \tau 2}\int_{\Om}  \int_{\Om}\frac{(\phi(x))^{\frac \tau 2}}{|x-y|^\mu}~dxdy + C_2^2 \int_{\Om} \int_{\Om} \frac{1}{|x-y|^{\mu}}~dxdy.
		\end{align*}
		From the last relation and \eqref{energy} , we obtain
		\begin{align}\label{j}
		J(t \phi) & \leq \|t \phi\|^{N}- \frac{1}{2}\int_{\Om}\left( \int_{\Om}\frac{F(y, h(t \phi))}{|x-y|^{\mu}} dy\right)F(x, h( t \phi))~dx\notag\\
		&\leq C_3 {t^N} -C_4 {t^\tau}
		+C_5 {t^{\frac \tau 2}} -{C_6},
		\end{align}
		where $C_i's$ are positive constants for $i=3,4,5,6$.
		%	 Now Lemma \ref{L1}-$(h_5)$ gives $$\lim_{t\to\infty}\frac{(h(t))^{2(d+1)}}{t^N}=\infty.$$ Also, $(h(t))^{2(d+1)}>(h(t))^{(d+1)}$  since $h$ is increasing. These facts combining with \eqref{j} imply that
		From \eqref{j}, we infer that   $J(t\phi) \to -\infty$ as $t \to \infty$, since $\tau>N$. Thus, there exists  $t_0(>0)\in\mathbb R$ so that $v_*(:=t_0 \phi)\in W^{1,N}_0(\Om)$ with $\|v_*\|> \rho_*$ such that $J(v_*)<0$.
	\end{proof}
	
	%%%%%%%%%%%%%%%%%%%%%%%%%%%%%%%%%%%%%%%%%%%%%%%%%%%%%%%%%%%%%%%%%%%%%%%%%%%%%%%%%%%%%%%%%%%%%%%%%%%%%%%%%%%%%%

	\noi From the above two lemmas, we get that $J$ satisfies  the mountain pass geometry near $0$.  Let $\ds \Gamma=\{\gamma\in C([0,1],W_0^{1,N}(\Om)):\gamma(0)=0,J(\gamma(1))<0\}$ and define the mountain pass critical level
	\begin{align}\label{level}\ds \beta^*=\inf_{\gamma\in \Gamma}\max_{t\in[0,1]}J(\gamma(t)).\end{align} Then by Lemma \ref{lem7.1}, Lemma \ref{lem7.1.2} and the mountain pass theorem we know that there exists a Palais-Smale sequence $\{w_k\}\subset W_0^{1,N}(\Om)$ for $J$ at level $\beta^*$, that is, as $k\to\infty$
	\[J(w_k) \to \beta^*; \; \text{and}\; J^\prime(w_k) \to 0 \text{\;\; in \;} \left(W_0^{1,N}(\Om)\right)^*,\] where $\left(W_0^{1,N}(\Om)\right)^*$ denotes the  dual of $W_0^{1,N}(\Om)$. Moreover, Lemma \ref{lem7.1} guarantees that $\beta^*>0.$
	%%%%%%%%%%%%%%%%%%%%%%%%%%%%%%%%%%%%%%%%%%%%%%%%%%%%%%%%%%%%%%%%%%%%%%%%%%%%%%%%%%%%%%%%%%%%%
	\subsection{Analysis of Palais-Smale sequence}
	\begin{lemma}\label{lem712}
		Let $(f_1)$ and $(f_5)$	hold. Let  $h$ be given as in \eqref{g}. Then every Palais-Smale sequence of $J$ is bounded in $W_0^{1,N}(\Om)$.
	\end{lemma}
	\begin{proof}
		Let $\{w_k\} \subset W_0^{1,N}(\Om)$ be a Palais-Smale sequence of $J$ at level $c\in\mb R$, that is, as $ k \to \infty$
		\begin{align*}
		J(w_k) \to c,  \; J^{\prime}(w_k) \to 0 \text{\; in } ( W_0^{1,N}(\Om))^*.
		\end{align*}
		 Then we have
		\begin{align}\label{kc-PS-bdd1}
		&\frac 1N\|w_k\|^N - \frac12 \int_{\Om} \left(\int_{\Om} \frac{F(y,h(w_k)}{|x-y|^{\mu}}dy \right)F(x,h (w_k))~dx \to c \; \text{as}\; k \to \infty,\notag\\
		&\left| \int_{\Om}{|\nabla w_k|^{N-2}\nabla w_k \nabla\phi}{dx}
		 -\int_\Om \left(\int_\Om \frac{F(y,h(w_k))}{|x-y|^{\mu}}dy \right)f(x,h(w_k)) h'(w_k)\phi ~dx \right|\leq \e_k\|\phi\|,
		\end{align}
		where $\e_k \to 0$ as $k\to \infty$. In the last relation, taking $\phi=w_k$ we get
		\begin{equation}\label{kc-PS-bdd2}
		\left| \|w_k\|^{N}-\int_\Om \left(\int_\Om \frac{F(y,h(w_k))}{|x-y|^{\mu}}dy \right)f(x,h(w_k) h'(w_k) w_k ~dx \right|\leq \e_k\|w_k\|.
		\end{equation}
		Now set $$v_k:=\frac{h(w_k)}{h'(w_k)}.$$  Since by Lemma \ref{L1}-$(h_4)$, we have $|v_k|\leq 2 |w_k|$ and by using \eqref{g} and Lemma \ref{L1}-$(h_5)$, we get 
		$$|\nabla v_k|= \left(1+\frac{2^{N-1}|h(w_k)|^N}{1+2^{N-1}|h(w_k)|^N}\right)|\nabla w_k|^N\leq 2|\nabla w_k|, $$	 combining these, we obtain
		$$\|v_k\|\leq2\|w_k\|.$$ Therefore, $v_k\in W^{1,N}_0(\Om)$. Then by  choosing $\phi=v_k$ and inserting this into \eqref{kc-PS-bdd1}, we deduce
		\begin{align}\label{1.0}
	\left|	\langle J'(w_k), v_k\rangle\right|&=\bigg|\int_\Om\left(1+\frac{2^{N-1}|h(w_k)|^N}{1+2^{N-1}|h(w_k)|^N}\right)|\nabla w_k|^N dx-\int_\Om \left(\int_\Om \frac{F(y,h(w_k))}{|x-y|^{\mu}}dy \right)f(x,h(w_k)) h(w_k) dx \bigg|\notag\\
		&\leq\e_k\|v_k\|\leq 2\e_k\|w_k\|,
		\end{align}	where $\langle\cdot,\cdot\rangle$ denotes the dual pairing of between $W_0^{1,N}(\Om)$ and its dual $\left(W_0^{1,N}(\Om)\right)^*.$
		Using  \eqref{1.0} and $(f_5)$,    we get
		\begin{align*}
		C+\e_k\|w_k\|&\geq J(w_k)-\frac{1}{2\tau}\langle J'(w_k), v_k\rangle\notag\\
		&\geq \int_\Om\left[\frac1N-\frac{1}{2 \tau}\left(1+\frac{2^{N-1}|h(w_k)|^N}{1+2^{N-1}|h(w_k)|^N}\right)\right]|\nabla w_k|^N dx\notag\\&\;\;\;\;\;\;-\frac12\int_\Om \left(\int_\Om \frac{F(y,h(w_k))}{|x-y|^{\mu}}dy \right)\left[F(x,h(w_k))-\frac1\tau f(x,h(w_k)) h(w_k)\right] dx\notag\\
		&\geq \int_\Om\left(\frac1N-\frac{1}{\tau}\right)|\nabla w_k|^N dx=\left(\frac1N-\frac{1}{\tau}\right)\|w_k\|^N.\notag
		\end{align*}
Hence, $\{w_k\}$ must be bounded in $W^{1,N}_0(\Om)$.
	\end{proof}
	\noi Next, we have the following lemma:
\begin{lemma}\label{PS-level}
		If $f$ satisfies $(f_1)$-$(f_6)$ and  $h$ is defined as in \eqref{g}, then 	{ $$0< \beta^* < \displaystyle \frac{1}{2N }\left( \frac{2N-\mu}{2N}\alpha_N\right)^{N-1},$$  } where $\beta^*$ is given as in \eqref{level}.
	\end{lemma}
	\begin{proof}
	{	By Lemma \ref{lem7.1.2}, for $w \in W_0^{1,N}(\Om)$ with $\|w\|=1,$ $J(tw)\to-\infty$ as $t\to+\infty$. Also,  from \eqref{level}, we get   $$\beta^* \leq \displaystyle\max_{t\in[0,1], w \in {W^{1,N}_0(\Omega)\setminus\{0\}}} J(tw)$$ with $J(w)<0.$ So, it is sufficient to prove that there exists some $v\in W^{1,N}_0(\Om)$ such that $\|v\|=1$ and
		\[\max_{t\in[0,\infty)} J(tv) < \frac{1}{2N} \left( \frac{2N-\mu}{2N}\alpha_N\right)^{N-1}.\]}
		To show this, let us consider the sequence of Moser functions $\{\mc M_k\}$ defined as
		\begin{equation*}
		\mc M_k(x)=\frac{1}{\omega_{N-1}^{\frac{1}{N}}}\left\{
		\begin{split}
		& (\log k)^{\frac{N-1}{N}},\; 0\leq |x|\leq \frac{\delta}{k},\\
		& \frac{\log \left(\frac{\delta}{|x|}\right)}{(\log k)^{\frac{1}{N}}}, \; \frac{\delta}{k}\leq |x|\leq \delta\\
		& 0,\; |x|\geq \delta.
		\end{split}
		\right.
		\end{equation*}
		Then it is easy to see that supp$(\mc M_k) \subset B_\delta(0)$ and $\|\mc M_k\| =1$ for all $k\in\mathbb N$. Now we assert that there exists some $k \in \mb N$ such that
		\begin{align}\label{key}\max_{t\in[0,\infty)} J(t\mc M_k) < \frac{1}{2N }\left( \frac{2N-\mu}{2N}\alpha_N\right)^{N-1}.\end{align}
		Indeed, if this doesn't hold  then for all $k \in \mb N,$ there exists  $t_k>0$ such that
		\begin{equation}\label{kc-PScond0}
		\begin{split}
		&\max_{t\in[0,\infty)} J(t\mc M_k) = J(t_k\mc M_k) \geq \frac{1}{2N} \left( \frac{2N-\mu}{2N}\alpha_N\right)^{N-1}.
		\end{split}
		\end{equation}
		%	From the proof of Lemma \ref{lem7.1}, $J(t w_k)\to -\infty$ as $t\to \infty$ uniformly in $k$. Then we infer that $\{t_k\}$ must be a bounded sequence in $\mb R$. 
		Since    $(f_1)$ implies  that $F(x,h(t_k\mc M_k))\geq 0$ for all $k\in\mathbb N$, by  the definition of $J(t_kw_k)$ together with  \eqref{kc-PScond0}, we obtain
		\begin{equation}\label{kc-PScond2}
		t_k^N \geq \frac{1}{2} \left( \frac{2N-\mu}{2N}\alpha_N\right)^{N-1}.
		\end{equation}
	Also, in view of \eqref{kc-PScond0}, we have
		$\frac{d}{dt}(J(t\mc M_k))|_{t=t_k}=0.$
		This combining with Lemma \ref{L1}-$(h_4)$ yields that
		\begin{align}\label{kc-PS-cond3}
		t_k^N &= \int_\Om \left(\int_\Om \frac{F(y,h(t_k\mc M_k))}{|x-y|^{\mu}}dy\right)f(x,h(t_k\mc M_k))h'(t_k\mc M_k)t_k\mc M_k ~dx\notag\\
		&\geq\frac 12 \int_{B_{\delta/k}(0)}f(x,h(t_k\mc M_k))h(t_k\mc M_k)\left( \int_{B_{\delta/k}(0)}\frac{F(y,h(t_k\mc M_k))}{|x-y|^\mu}~dy\right)~dx.
		%& \geq (\beta-\e)\exp\left( a \left( \frac{t_k(\log k)^{\frac{n-1}{n}}}{\omega_{n-1}^{\frac{1}{n}}}\right)^{\frac{n}{n-1}}\right)\int_{B_{\rho/k}}\int_{B_{\rho/k}} \frac{~dxdy}{|x-y|^\mu}
		\end{align}
		Now from \eqref{h-growth}, we know that for each $b>0$ there exists a constant $R_b$ such that
		\[sf(x,s)F(x,s) \geq b \exp\left( 2|s|^{\frac{2N}{N-1}}\right),\; \text{whenever}\; s \geq R_b.\]
		From \eqref{kc-PScond2}, we infer that
		${t_k}\mc M_k \to \infty \;\text{as}\; k \to \infty$ in $B_{\de/k}(0).$ Now by Lemma \ref{L1}-$(h_6)$, we get $h({t_k}\mc M_k) \to \infty \;\text{as}\; k \to \infty$, uniformly   in $B_{\de/k}(0).$   So, we can choose $s_b \in \mb N$ such that  in $B_{\de/k}(0)$, 
		\[h(t_k\mc M_k) \geq R_b,\; \text{for all}\; k\geq s_b.\] Furthermore, by Lemma \ref{L1}-$(h_7)$,  for any  $\e>0$, there exists $k_0\in\mb N$ such that for all $k\geq k_0,$ 
		\begin{equation}\label{new1} |h({t_k}\mc M_k)|^{\frac{2N}{N-1}}\geq ({t_k}\mc M_k)^{\frac{N}{N-1}} (2^{\frac{1}{N-1}}-\e)\end{equation}
		In addition, using the same idea as in \cite{yang-JDE} (see  equation $(2.11)$), we can have
		\begin{equation} \label{new21}
			\int_{B_{\delta/k}(0)}\int_{B_{\delta/k}(0)} \frac{~dxdy}{|x-y|^\mu} \geq C_{\mu, N} \left(\frac{\delta}{k}\right)^{2N-\mu},
		\end{equation}
		where $C_{\mu, N}$ is a positive constant depending on $\mu$ and $N$.
		Using these estimates in \eqref{kc-PS-cond3}, \eqref{new1} and \eqref{new21} for sufficiently large $b$ and for $k\geq k_0$,
		we get 
	{	\begin{align}
		t_k^N&\geq \frac b2 \int_{B_{\delta/k}(0)}\int_{B_{\delta/k}(0)} \frac{~dxdy}{|x-y|^\mu}\exp \left( {2 (h(t_k\mc M_k))^{\frac{2N}{N-1}}}\right)\notag\\&\geq \frac b2 \int_{B_{\delta/k}(0)}\int_{B_{\delta/k}(0)} \frac{~dxdy}{|x-y|^\mu}\exp \left( {2 (2^{\frac{1}{N-1}}-\e)(t_k\mc M_k)^{\frac{N}{N-1}}}\right)\notag\\&=\frac b2 \exp \left( \log k\left(\frac{2 (2^{\frac{1}{N-1}}-\e)  t_k^{\frac{N}{N-1}}}{\omega_{N-1}^{\frac{1}{N-1}}}\right)\right)\int_{B_{\delta/k}(0)}\int_{B_{\delta/k}(0)}\frac{~dxdy}{|x-y|^\mu}\notag\\&\geq\frac b2 \exp \left( \log k\left(\frac{2 (2^{\frac{1}{N-1}}-\e) t_k^{\frac{N}{N-1}}}{\omega_{N-1}^{\frac{1}{N-1}}}\right)\right)C_{\mu,N}\left(\frac{\delta}{k}\right)^{2N-\mu}\notag\\&=\frac b2 \exp \left( \log k\left[\left(\frac{2 (2^{\frac{1}{N-1}}-\e) t_k^{\frac{N}{N-1}}}{\omega_{N-1}^{\frac{1}{N-1}}}\right)-(2N-\mu)\right]\right)C_{\mu,N}{\delta}^{2N-\mu}\notag.
		\end{align} Now taking $\e\to0$ in the above expression and using  \eqref{kc-PScond2}, we obtain
	\begin{align}\label{bd}
	t_k^N\geq \frac b2  C_{\mu, N}\delta^{2N-\mu},
	\end{align} 
uniformly in $k\in\mb N.$
		%which implies
		%\begin{equation}
		%\frac{m(t_k^n)t_k^n}{k^{{a t_k^{\frac{n}{n-1}}}{\omega_{n-1}^{-\frac{1}{n-1}}}-(4-\mu)}}\geq (\beta -\e)C_\mu\rho^{4-\mu}.
		%\end{equation}
		On the other hand, from the proof of Lemma \ref{lem7.1.2}, it follows that $J(t \mc M_k)\to-\infty$ as $t\to\infty$ uniformly in $k\in \mathbb N.$ Using this with  \eqref{kc-PScond0}, we get that the sequence  $\{t_k\}$ is bounded, which contradicts	\eqref{bd}, since $b$ is arbitrary.} This establishes our claim \eqref{key} and hence, we conclude the proof. \end{proof}
	\subsection{Convergence results}
	Let $\{w_k\}\subset W^{1,N}_0(\Om)$ be a Palais-Smale sequence for $J$.
	Now Lemma \ref{lem712} yields that $\{w_k\}$  is bounded in $W^{1,N}_0(\Om)$. Thus, there exists  $w\in W^{1,N}_0(\Om)$ such that up to a subsequence, still denoted by $\{w_k\}$, 
%	by Sobolev embedding theorem, we have 
	%\begin{align*}
	$w_k \rightharpoonup w$ weakly in $ W^{1,N}_0(\Om),
	w_k \to w$  strongly in $ L^q(\Om),\, q \in [1,\infty),
	w_k(x) \to w(x)$ point-wise a.e. in $ \Om,$
	%\end{align*} 
	as $k \to \infty$. Also, from \eqref{kc-PS-bdd1} and \eqref{kc-PS-bdd2}, we obtain that for some positive constants $C', C''$, 
	\begin{align}
	\int_\Om \left(\int_\Om\frac{F(y,h(w_k))}{|x-y|^\mu}dy\right)F(x,h(w_k))~dx &\leq C',\label{wk-sol10}\\
	\int_\Om \left(\int_\Om\frac{F(y,h(w_k))}{|x-y|^\mu}dy\right)f(x,h(w_k))h'(w_k)w_k~dx & \leq C''.\label{wk-sol100}
	\end{align}
	Now we have the next two results, where we consider the Palais-Smale sequence $\{w_k\}$ to be satisfying all the above  facts.  
	\begin{lemma}\label{PS-ws} 	 Assume that the assumptions $(f_1)$-$(f_5)$ hold and let $h$ be defined as in \eqref{g}. Let $\{w_k\}\subset W^{1,N}_0(\Om)$ be a Palais-Smale sequence for $J$.
		%	If $\{w_k\}\subset W^{1,N}_0(\Om)$ is a Palais-Smale sequence of $J$ then there exists  $w \in W^{1,N}_0(\Om)$ such that, up to a subsequence, $w_k \rightharpoonup w$ weakly in $W^{1,N}_0(\Om)$ and moreover,
		Then we have	\begin{align}\label{3.23}
		\lim_{k \to \infty} \int_{\Omega} \left(\int_{\Omega}\frac{|F(y,h(w_k))-F(y,h(w))|}{|x-y|^{\mu}} dy \right) |F(x,h(w_k))-F(x,h(w))|  ~dx =0.
		\end{align}
	\end{lemma}
	\begin{proof}
		%	We argue as along equation $(2.20)$ in Lemma $2.4$ in \cite{yang-JDE}. 
		%%%%%%%%%%%%%%%%%%%%%%%%%%%%%%%%%
	 We know that if a function $\mc F\in L^1(\Om)$ then for any $\e>0$ there exists a $\delta(\e)>0$ such that
		$\left| \int_{\Om'} \mc F(x)~dx\right| <\e,$
		for any measurable set $\Om'\subset \Om$ with $|\Om'|\leq \delta(\e)$.
		%Now Since $\left(\int_\Om\frac{G(y,u_k)}{|x-y|^\mu}dy\right)G(x,u_k)~dx \in L^{1}(\Om)$. It follows that for every $\e>0$, there exists $\de>0$ such that
		%\[\int_{U} \left(\int_\Om\frac{G(y,u_k)}{|x-y|^\mu}dy\right)G(x,u_k)~dx \leq \e\;\forall\; |U|\leq \de\]
		%for all measurable subset $U$ of $\Om$.
		Also, if $\mc F\in L^{1}(\Om)$ then for any fixed $\delta_*>0$ there exists $\al>0$ such that
		$$|\{x\in \Om : |\mc F(x)|\geq \al\}|\leq \delta_*.$$
		Now  \eqref{wk-sol10} gives that
		$$\left(\displaystyle\int_\Om\frac{F(y,h(w_k))}{|x-y|^\mu}dy\right)F(\cdot,h(w_k)) \in L^{1}(\Om)$$
		and similarly, \eqref{k2} gives that
		$$\left(\displaystyle\int_\Om\frac{F(y,h(w))}{|x-y|^\mu}dy\right)F(\cdot,h(w)) \in L^{1}(\Om).$$
		Thus we fix $\de_*>0$ and  choose $\al> \max\left\{1,\left(\frac{2C'' M_0}{{\de_*}}\right)^{\frac{1}{m_0+1}}, s_0\right\}$,  where $C''$ is given in \eqref{wk-sol100}, such that
		{\begin{align}\label{bb}
				&\int_{\Om \cap\{h(w)\geq \al\}} \left(\int_{\Omega} \frac{F(y,h(w))}{|x-y|^{\mu}} dy \right)F(x,h(w))  ~dx  \leq \de_*.
		\end{align}} Now by using $(f_4)$, Lemma \ref{L1}-$(h_4)$ and \eqref{wk-sol10}, we deduce
		\begin{align}\label{b}
		&\int_{\Om \cap\{h(w_k)\geq \al\}} \left(\int_{\Omega} \frac{F(y,h(w_k))}{|x-y|^{\mu}} dy \right)F(x,h(w_k))  ~dx \notag\\&\leq M_0 \int_{\Om \cap\{h(w_k)\geq \al\}} \left(\int_{\Omega} \frac{F(y,h(w_k))}{|x-y|^{\mu}}dy\right) \frac{f(x,h(w_k))}{(h(w_k))^{m_0}}  ~dx\notag\\
		&\leq M_0 \int_{\Om \cap\{h(w_k)\geq \al\}} \left(\int_{\Omega} \frac{F(y,h(w_k))}{|x-y|^{\mu}}dy\right) \frac{f(x,h(w_k))h(w_k)}{(h(w_k))^{m_0+1}}  ~dx\notag\\
		%&\leq {2M_0}\int_{\Om \cap\{h(w_k)\geq \al\}} \left(\int_{\Omega} \frac{F(y,h(w_k))}{|x-y|^{\mu}} dy\right) \frac {f(x,h(w_k)) h'(w_k)w_{k}}{(C^{*}w_k)^{m_0+1}}  ~dx<\\
		&\leq \frac{2M_0}{ \al^{m_0+1}}\int_{\Om \cap\{h(w_k)\geq \al\}} \left(\int_{\Omega} \frac{F(y,h(w_k))}{|x-y|^{\mu}} dy\right) {f(x,h(w_k)) h'(w_k)w_{k}} ~dx<\de_*.
		\end{align}
		
	Combining \eqref{b} and \eqref{bb}, we have
		\begin{align*}
		&\left| \int_{\Omega} \left(\int_{\Omega}\frac{F(y,h(w_k))}{|x-y|^{\mu}} dy \right) F(x,h(w_k)) ~dx-  \int_{\Omega} \left(\int_{\Omega}\frac{F(y,h(w) )}{|x-y|^{\mu}} dy \right) F(x,h(w)) ~dx\right|\\
		&\leq 2\de_*+ \bigg|\int_{\Om \cap \{h(w_k)\leq \al\}} \left(\int_{\Omega}\frac{F(y,h(w_k))}{|x-y|^{\mu}} dy \right) F(x,h(w_k)) ~dx\\&\qquad\qquad\qquad-  \int_{\Omega\cap \{h(w)\leq \al\}} \left(\int_{\Omega}\frac{F(y,h(w))}{|x-y|^{\mu}} dy \right) F(x,h(w)) ~dx\bigg|.
		\end{align*}
		Next, we show  that as $k \to \infty$
		\begin{align*}
		\int_{\Om \cap \{h(w_k)\leq \al\}} \left(\int_{\Omega}\frac{F(y,h(w_k))}{|x-y|^{\mu}} dy \right) F(x,h(w_k))~dx \to \int_{\Om \cap \{h(w)\leq \al\}} \left(\int_{\Omega}\frac{F(y,h(w))}{|x-y|^{\mu}} dy \right) F(x,h(w)) ~dx.
		\end{align*}
		Since $\left(\displaystyle\int_\Om\frac{F(y,h(w))}{|x-y|^\mu}dy\right)F(\cdot,h(w)) \in L^{1}(\Om)$,  Fubini's theorem yields that
		\begin{align*}
		%\lim_{K \to \infty} \int_{|u_k|\leq M}\left(\int_{|u_k|\geq K}\frac{G(x,u_k)}{|x-y|^{\mu}}dy\right)G(x,u)~dx =
		&\lim_{\Lambda \to \infty} \int_{\Om \cap\{h(w)\leq \al\}}\left(\int_{\Om\cap\{h(w)\geq \Lambda\}}\frac{F(y,h(w))}{|x-y|^{\mu}}dy\right)F(x,h(w))~dx\\
		&= \lim_{\Lambda \to \infty} \int_{\Om \cap\{h(w)\geq \Lambda\}}\left(\int_{\Om \cap\{h(w)\leq \al\}}\frac{F(y,h(w))}{|x-y|^{\mu}}dy\right)F(x,h(w))~dx=0.
		\end{align*}
		Therefore, we can fix  $\Lambda> \max\left\{\left(\frac{2C'' M_0}{\de_*}\right)^{\frac{1}{m_0+1}}, s_0\right\}$, where $C''$ is given in \eqref{wk-sol100},  such that
		\[\int_{\Om \cap\{h(w)\leq \al\}} \left(\int_{\Om \cap\{h(w)\geq \Lambda\} }\frac{F(y,h(w) )}{|x-y|^{\mu}} dy \right) F(x,h(w)) ~dx \leq \delta_*.\] Next, using \eqref{wk-sol10}, $(f_4)$ and Lemma \ref{L1}-$(h_4)$, we deduce
		\begin{align*}
		&\int_{\Om \cap\{h(w_k)\leq \al\}} \left(\int_{\Om \cap\{h(w_k)\geq \Lambda\}}\frac{F(y,h(w_k))}{|x-y|^{\mu}} dy \right)F(x,h(w_k)) ~dx\\
		& \leq M_0\int_{\Om \cap \{h(w_k)\leq \al\}} \left(\int_{\Om \cap\{h(w_k)\geq \Lambda\}}\frac{ f(y,h(w_k))}{(h(w_k))^{m_0}|x-y|^{\mu}} dy \right)F(x,h(w_k)) ~dx\\
		%&\leq {M_0} \int_{\Om \cap\{h(w_k)\leq \al\}} \left(\int_{\Om \cap\{h(w_k)\geq \Lambda\} }\frac{  f(y,h(w_k))h(w_k)}{( h(1)w_k(y))^{m_0+1}|x-y|^{\mu}} dy \right) F(x,h(w_k)) ~dx\\
		&\leq \frac{M_0}{{\Lambda}^{m_{0}+1}} \int_{\Om \cap\{h(w_k)\leq \al\}} \left(\int_{\Om \cap\{h(w_k)\geq \Lambda\} }\frac{  f(y,h(w_k))h (w_k)}{|x-y|^{\mu}} dy \right) F(x,h(w_k)) ~dx\\
		&\leq \frac{2M_0}{\Lambda^{m_0+1}} \int_{\Om \cap\{h(w_k)\leq \al\}} \left(\int_{\Om \cap\{h(w_k)\geq \Lambda\} }\frac{  f(y,h(w_k))h'(w_k) w_k(y)}{|x-y|^{\mu}} dy \right) F(x,h(w_k)) ~dx\\
		&\leq \frac{2M_0}{\Lambda^{m_{0}+1}} \int_{\Om} \left(\int_{\Om }\frac{F(y,h(w_k))}{|x-y|^{\mu}} dy \right) f(x,h(w_k)) h'(w_k)w_k ~dx\leq \de_*.
		\end{align*}
		
		Thus, we obtain
		\begin{align*}
		&\left|\int_{\Om \cap\{h(w)\leq \al\}} \left(\int_{\Om \cap\{h(w)\geq \Lambda\} }\frac{F(y,h(w) )}{|x-y|^{\mu}} dy \right) F(x,h(w)) ~dx\right.\\
		&\quad \quad \left.- \int_{\Om \cap\{h(w_k)\leq \al\}} \left(\int_{\Om \cap\{h(w_k)\geq \Lambda\} }\frac{F(y,h(w_k))}{|x-y|^{\mu}} dy \right) F(x,h(w_k)) ~dx\right|\leq 2\de_*.
		\end{align*}
		Now we claim that as $k\ra \infty$, for fixed positive real numbers $\al$ and $\Lambda$, the following holds:
		\begin{equation}\label{choq-new}
		\begin{split}
		&\left|\int_{\Om\cap\{h(w_k)\leq \al\}} \left(\int_{\Om \cap\{h(w_k)\leq \Lambda\}}\frac{F(y,h(w_k))}{|x-y|^{\mu}} dy \right) F(x,h(w_k)) ~dx\right.\\
		&\qquad- \left.\int_{\Om \cap\{h(w)\leq \al\}} \left(\int_{\Om \cap\{h(w)\leq \Lambda\}}\frac{F(y,h(w) )}{|x-y|^{\mu}} dy \right) F(x,h(w)) ~dx \right|\ra 0.
		\end{split}
		\end{equation}
		It is easy to compute that
		\begin{align}\label{a0}
		&\left(\int_{\Om \cap\{h(w_k)\leq \Lambda\} }\frac{F(y,h(w_k))}{|x-y|^{\mu}} dy \right) F(x,h(w_k))\chi_{ \Om \cap \{h(w_k)\leq \al\}}\notag\\  &\qquad\ra \left(\int_{\Om \cap \{h(w)\leq \Lambda\} }\frac{F(y,h(w))}{|x-y|^{\mu}} dy \right) F(x,h(w))\chi_{\Om \cap \{h(w)\leq \al\}}
		\end{align}
		point-wise a.e. as $k \to \infty$. Now taking $N=r$ in \eqref{k1} and using Lemma \ref{L1}-$(h_5)$ and \eqref{HLSineq}, we get a constant $C_{\al,\Lambda}>0$ depending on $\al$ and $\Lambda$ such that
		\begin{align*}
		&\int_{\Om \cap \{h(w_k)\leq \al\}}\left( \int_{\Om \cap \{h(w_k)\leq \Lambda\} }\frac{F(y,h(w_k))}{|x-y|^{\mu}} dy \right)  F(x,h(w_k))dx  \notag\\
		&\leq  C_{\al,\Lambda}\int_{\Om \cap \{h(w_k)\leq \al\}}\left( \int_{\Om\cap\{h(w_k)\leq \Lambda\} }\frac{|h(w_k(y))|^{N}}{|x-y|^{\mu}} dy \right)  |h(w_k(x))|^{N} dx \notag\\
		&\leq  C_{\al,\Lambda}\int_{\Om \cap \{h(w_k)\leq \al\}}\left( \int_{\Om\cap\{h(w_k)\leq \Lambda\} }\frac{|w_k(y)|^{N}}{|x-y|^{\mu}} dy \right)  |w_k(x)|^{N} dx \notag\\
%		& \leq C_{\al,\Lambda} \int_\Om\int_{\Om }\left(\frac{|w_k(y)|^{N}}{|x-y|^{\mu}}~dy  \right) |w_k(x)|^{N} ~dx\notag\\
		& \leq { {C_{\al,\Lambda}C(N,\mu)\|w_k\|_{L^{\frac{2N^2}{2N-\mu}}(\Om)}^{2N} \to C_{\al,\Lambda}C(N,\mu)\|w\|_{L^{\frac{2N^2}{2N-\mu}}(\Om)}^{2N}}} \; \text{as}\; k \to \infty,
		%&\ra \left(\int_{\Om } \frac{|u(y)|^{p+1}}{|x-y|^{\mu}} dy \right) u^{p+1}(x)\chi_{\Om} ~dx
		\end{align*}
	 since $w_k \to w$ strongly in $L^q(\Om)$ for each $q \in [1,\infty)$.  This combining with Theorem $4.9$ in \cite{bz} implies that there exists $\mc G \in L^1(\Om)$ such that  up to a subsequence, for each $k\in \mathbb N$, we have 
		\[\left|\left( \int_{\Om \cap\{h(w_k)\leq \Lambda\} }\frac{F(y,h(w_k))}{|x-y|^{\mu}} dy \right)  F(x,h(w_k))\chi_{ \Om \cap \{h(w_k)\leq \al\}} \right| \leq |\mc G(x)|.\]
		Therefore, using \eqref{a0} and applying the Lebesgue dominated convergence theorem, we obtain \eqref{choq-new} and hence, as $k\to\infty,$ we get
		\begin{align}\label{FF}
	 \int_{\Omega} \left(\int_{\Omega}\frac{F(y,h(w_k))}{|x-y|^{\mu}} dy \right) F(x,h(w_k)) ~dx\to  \int_{\Omega} \left(\int_{\Omega}\frac{F(y,h(w) )}{|x-y|^{\mu}} dy \right) F(x,h(w)) ~dx.
		\end{align} 
Following the similar arguments as discussed above in this proof, we can find as $\Lambda \to \infty$
		\begin{align}\label{3.25}
		&\int_{\Omega} \int_{\Om\cap\{h(w) \geq \Lambda\}} \frac{F(y,h(w))}{|x-y|^{\mu}} F(x,h(w)) dy ~dx=o(\Lambda),\\ &\int_{\Omega} \int_{\Om\cap\{h(w_k) \geq \Lambda\}} \frac{F(y,h(w_k))}{|x-y|^{\mu}} F(x,h(w_k)) dy ~dx = o(\Lambda),
		\end{align}
		\begin{equation}
		\int_{\Omega} \int_{\Om\cap \{h(w) \geq \Lambda\}} \frac{F(y,h(w_k))}{|x-y|^{\mu}} F(x,h(w)) dy ~dx = o(\Lambda),
		\end{equation}
		and
		\begin{equation}\label{3.27}
		\int_{\Omega} \int_{\Om\cap\{h(w_k) \geq \Lambda\}} \frac{F(y,h(w_k))}{|x-y|^{\mu}} F(x,h(w)) dy ~dx = o(\Lambda).
		\end{equation}
	So,
		\begin{equation*}
		\begin{split}
		& \int_{\Omega} \left(\int_{\Omega}\frac{|F(y,h(w_k))-F(y,h(w))|}{|x-y|^{\mu}} dy \right) |F(x,h(w_k))-F(x,h(w))|  ~dx\\&\leq
		2\int_{\Omega} \left(\int_{\Omega}\frac{\chi_{\{h(w_k)\geq \Lambda\}}(y)F(y,h(w_k))}{|x-y|^{\mu}} dy \right) F(x,h(w_k)) ~dx \\
		&+4 \int_{\Omega} \left(\int_{\Omega}\frac{F(y,h(w_k))\chi_{\{h(w)\geq \Lambda\}}(x)F(x,h(w))}{|x-y|^{\mu}} dy \right) ~dx\\&+4 \int_{\Omega} \left(\int_{\Omega}\frac{\chi_{\{h(w_k)\geq \Lambda\}}(y)F(y,h(w_k))F(x,h(w))}{|x-y|^{\mu}} dy \right) ~dx\\
		&+2\int_{\Omega} \left(\int_{\Omega}\frac{\chi_{\{h(w)\geq \Lambda\}}(y)F(y,h(w))}{|x-y|^{\mu}} dy \right) F(x,h(w)) ~dx \\
		&+\int_\Omega\Bigg[\left(\int_\Omega\frac{|F(y,h(w_k))\chi_{\{h(w_k)\leq \Lambda\}}(y)-F(y,h(w))\chi_{\{h(w)\leq \Lambda\}}(y)|}{|x-y|^\mu}dy\right)\\&\qquad\qquad\qquad|F(x,h(w_k))\chi_{\{h(w_k)\leq \La\}}(x)-F(x,h(w))\chi_{\{h(w)\leq \Lambda\}}(x)|\Bigg]~dx.
		\end{split}
		\end{equation*}
		Then employing the Lebesgue dominated convergence theorem we infer that the last integrand tends to $0$ as $k \to \infty.$ Hence, making use of \eqref{3.25}-\eqref{3.27}, we finally can conclude \eqref{3.23}.	
		\end{proof}
		\begin{lemma}\label{kc-ws} Assume that $(f_1)$-$(f_5)$ hold and  $h$ is defined as in \eqref{g}. Let $\{w_k\}\subset W^{1,N}_0(\Om)$ be a Palais-Smale sequence for $J$. Then
		for all $\varphi\in W_0^{1,N}(\Om),$   we have
		\begin{align*}\displaystyle\lim_{k\to\infty}\int_\Om \left(\int_\Om \frac{F(y, h(w_k))}{|x-y|^{\mu}}dy\right)f(x, h(w_k))h'(w_k)\varphi ~dx= \int_\Om \left(\int_\Om \frac{F(y,h(w))}{|x-y|^{\mu}}dy\right)f(x,h(w))h'(w)\varphi~dx.  \end{align*}
	\end{lemma}
	\begin{proof} 
		%	If $\{w_k\}$ is a Palais-Smale sequence at level $\beta^*$ for $J$ then  \eqref{kc-PS-bdd1} and \eqref{kc-PS-bdd2}hold. We remark that $J(u^+)\leq J(u)$ for each $u \in W^{1,N}_0(\Om)$, then we can assume $w_k \geq 0$ for each $k \in \mb N$. From Lemma \ref{lem712} we know that $\{w_k\}$ must be bounded in $W^{1,N}_0(\Om)$ so there exists a $C_*>0$ such that $\|w_k\|\leq C_*$. Also there exists a $w_0 \in W^{1,N}_0(\Om)$ such that up to a subsequence $w_k \rightharpoonup w_0$ in $W^{1,N}_0(\Om)$, strongly in $L^{q}(\Om)$ for all $q \in [1,\infty)$ and point-wise a.e. in $\Om$ as $k \to \infty$.
		Let $\Om' \subset\subset \Om$ and $\psi \in C_c^\infty(\Om)$ such that $0\leq \psi \leq 1$ and $\psi \equiv 1$ in $\Om' $. One can easily compute that
		\begin{equation}\label{k}
		\begin{split}
		\left\| \frac{\psi}{1+w_k}\right\|^N = \int_\Om \left|\frac{\nabla \psi}{1+w_k}- \psi \frac{\nabla w_k}{(1+w_k)^2} \right|^N~dx
		\leq 2^{N-1}(\|\psi\|^N+ \|w_k\|^N),
		\end{split}
		\end{equation}
		which yields that $\frac{\psi}{1+w_k} \in W^{1,N}_0(\Om)$. Now taking $\phi=\frac{\psi}{1+w_k}$  in \eqref{kc-PS-bdd1} as a test function and using Lemma \ref{L1}-$(h_3)$ and \eqref{k}, we obtain
		\begin{align}\label{kc-ws-new1}
		&\int_{\Om^{'}}\left( \int_\Om \frac{F(y,h(w_k))}{|x-y|^\mu}dy\right)\frac{f(x,h(w_k))}{1+w_k}h'(w_k)~dx\notag\\& \leq \int_\Om \left( \int_\Om \frac{F(y,h(w_k))}{|x-y|^\mu}dy\right)\frac{f(x,h(w_k)) h'(w_k)\psi}{1+w_k}~dx\notag\\
		&\leq \int_\Om \left( \int_\Om \frac{F(y,h(w_k))}{|x-y|^\mu}dy\right)\frac{f(x,h(w_k)) \psi}{1+w_k}~dx\notag\\
		&=  \e_k \left\|\frac{\psi}{1+w_k}\right\| + \int_\Om  |\nabla w_k|^{N-2}\nabla w_k \nabla \left( \frac{\psi}{1+w_k}\right)~dx\notag\\
		& \leq \e_k 2^{\frac{N-1}{N}}(\|\psi\|+ \|w_k\|) +\int_\Om |\nabla w_k|^{N-2}\nabla w_k \left(\frac{\nabla \psi}{1+w_k}-\psi\frac{\nabla w_k}{(1+w_k)^2}\right)~dx\notag\\
		& \leq \e_k 2^{\frac{N-1}{N}}(\|\psi\|+ \|w_k\|) +\int_\Om |\nabla w_k|^{N-1} \left( |\nabla \psi|+ |\nabla w_k|\right)~dx\notag\\
		& \leq \e_k 2^{\frac{N-1}{N}}(\|\psi\|+ \|w_k\|)+ [\|\psi\|\|w_k\|^{N-1}+ \|w_k\|^N]\leq C_1,
		\end{align}
		where $C_1$ is a positive constant and  in the last line we used the fact that $\{w_k\}$ is bounded in $W^{1,N}_0(\Om)$.
		% we infer that there exists some constant $C_1>0$ such that
		%	\begin{equation}\label{kc-ws-new1}
		%	\int_{\Om^{'}} \left(\int_\Om  \frac{F(y,h(w_k))}{|x-y|^\mu}dy\right)\frac{f(x,h(w_k))h'(w_k)}{1+w_k}~dx  
		%	\end{equation}
		Again, using the boundedness of the sequence $\{w_k\}$, from \eqref{kc-PS-bdd2}, we get
		\begin{align}\label{kc-ws-new2}
		\int_{\Om^{'}} \left( \int_\Om \frac{F(y,h(w_k))}{|x-y|^\mu}dy\right){f(x,h(w_k))h'(w_k)}{w_k}~dx&
		\leq \int_{\Om} \left( \int_\Om \frac{F(y,h(w_k))}{|x-y|^\mu}dy\right){f(x,h(w_k))h'(w_k)}{w_k}~dx \notag\\
		&\leq \e_k\|w_k\|+\|w_k\|^N\leq C_2
		\end{align}
		for some constant $C_2>0$. Combining \eqref{kc-ws-new1} and \eqref{kc-ws-new2}, we deduce
		\begin{align*}
		&\int_{\Om^{'}}\left( \int_\Om \frac{F(y,h(w_k))}{|x-y|^\mu}dy\right){f(x,h(w_k))h'(w_k)}~dx \\
		& \leq  2\int_{\Om^{'}\cap \{w_k <1\}} \left( \int_\Om\frac{F(y,h(w_k))}{|x-y|^\mu}dy\right)\frac{f(x,h(w_k))h'(w_k)}{1+w_k}~dx\\& + \int_{\Om^{'}\cap \{w_k \geq 1\}} \left(\int_\Om \frac{F(y,h(w_k))}{|x-y|^\mu}dy\right)w_k{f(x,h(w_k))h'(w_k)}~dx\\
		& \leq 2\int_{\Om^{'}} \left( \int_\Om\frac{F(y,h(w_k))}{|x-y|^\mu}dy\right)\frac{f(x,h(w_k))h'(w_k)}{1+w_k}~dx\\& + \int_{\Om^{'}}\left( \int_\Om \frac{F(y,h(w_k))}{|x-y|^\mu}dy\right)w_k{f(x,h(w_k))h'(w_k)}~dx\\
		& \leq 2C_1+C_2  :=C_3.
		\end{align*}
		Thus, the sequence $\{v_k\}:=\left\{\left( \int_\Om\frac{F(y,h(w_k))}{|x-y|^\mu}dy\right){f(x,h(w_k))h'(w_k)}\right\}$ is bounded in $L^1_{\text{loc}}(\Om)$. Therefore, there exists a radon measure $\zeta$ such that, up to a subsequence, $v_k \rightharpoonup \zeta$ in the ${weak}^*$-topology as $k \to \infty$.  Hence,  we have
		\[\lim_{k \to \infty}\int_\Om\int_\Om \left( \frac{F(y,h(w_k))}{|x-y|^\mu}dy\right){f(x,h(w_k))h'(w_k)}\eta ~dx = \int_\Om \eta ~d\zeta,\; \forall \eta \in C_c^\infty(\Om). \]
		Since $w_k$ satisfies \eqref{kc-PS-bdd1}, we achieve
		\[\int_A \eta d\zeta= \lim_{k \to \infty} \int_A |\nabla w_k|^{N-2}\nabla w_k \nabla \eta ~dx, \;\;\forall A\subset \Om,  \]
		which together with Lemma~\ref{wk-sol} yields that the Radon measure $\zeta$ is absolutely continuous with respect to the Lebesgue measure. So, there exists a function $\varrho \in L^1_{\text{loc}}(\Om)$ such that for any $\eta\in C^\infty_c(\Omega)$, it holds that $\int_\Om \eta~ d\zeta= \int_\Om \eta \varrho~dx$, thanks to Radon-Nikodym theorem,.
		Therefore, we obtain
		\begin{align*}&\lim_{k \to \infty}\int_\Om\left( \int_\Om \frac{F(y,h(w_k))}{|x-y|^\mu}dy\right){f(x,h(w_k))h'(w_k)}\eta(x)~ ~dx\\&\qquad = \int_\Om \eta \varrho~dx= \int_\Om  \left( \int_\Om \frac{F(y,h(w))}{|x-y|^\mu}dy\right){f(x,h(w))h'(w)}\eta(x)~ ~dx, \;\;\forall\; \eta\in C^\infty_c(\Omega), \end{align*}
		since $ C^\infty_c(\Omega)$ is dense in $W_0^{1,N}(\Om)$, this completes the proof.
	\end{proof}
	
	%\subsection{Existence of ground-state solution}
	%First we define the Nehari manifold associated to \eqref{pp} as
	%\[\mc N = \{u \in W^{1,N}_0(\Om)\setminus \{0\}:\; \langle J^\prime(u),u \rangle=0\}\]
	%and $\beta^{**} = \inf_{u\in \mc N}J(u)$.
	%Next, we recall the following higher integrability lemma from \cite{Lions}:
	%\begin{lemma}\label{Lions-lem}
	%	Let $\{u_k \in W^{1,N}_0(\Om):\; \|u_k\|=1\}$ be a sequence in $W^{1,N}_0(\Om)$ converging weakly to a non zero $v \in W^{1,N}_0(\Om)$. Then for every $p\in \left(1, (1-\|u\|)^{-\frac{1}{N-1}}\right)$,
	%	\[\sup_k \int_\Om \exp \left( p\alpha_N |u_k|^{\frac{N}{N-1}}\right)<+\infty. \]
	%\end{lemma}
	%\begin{proposition}\label{nt}
	%The problem \eqref{pp} achieves a nontrivial positive weak solution.
	%\end{proposition}
	%\begin{proof}
	\noi \textbf{Proof of Theorem \ref{T1}:}
	Let $\{w_k\}$ be a Palais-Smale sequence at the level $\beta^*$. Then $\{w_k\}$ can be obtained as a minimizing sequence associated to the variational problem \eqref{level}. Then by Lemma \ref{lem712},  there exists  $w \in W^{1,N}_0(\Om)$ such that, up to a subsequence,
	$w_k \rightharpoonup w$ weakly in $W^{1,N}_0(\Om)$ as $k \to \infty$.\\ % We know that $u_k \rightharpoonup u_0$ weakly in $W^{1,n}_0(\Om)$ and
	%\[\int_\Om \left(\int_\Om \frac{ F(y,u_k)}{|x-y|^{\mu}}dy\right)f(x,u_k)u_k ~dx \leq C \int_\Om f(x,u_k)u_k ~dx.  \]
	Now by using  Lemma \ref{wk-sol} and Lemma \ref{kc-ws}, we infer that $w$ forms a weak solution of \eqref{pp}.
	We claim that $w \not\equiv 0$. Indeed, if not, that is,  if $w \equiv 0$ then using \eqref{FF}, we have
	\[\int_\Om \left(\int_\Om \frac{ F(y,h(w_k))}{|x-y|^{\mu}}dy\right)F(x,h(w_k)) ~dx  \to 0\; \text{as}\; k \to \infty,\]
	which yields that $ \displaystyle\lim_{k \to \infty}J(w_k) = \frac{1}{N}\displaystyle\lim_{k\to\infty} \|w_k\|^N = \beta^*.$ That is,
	$$\lim_{k\to\infty} \|w_k\|^N = \beta^*N.$$ Therefore, recalling Lemma \ref{PS-level},  there exists a real number $l>0$, very close to $1$, and  there exists $k_0\in\mathbb N$ depending on $l$ such that 
	\begin{align}\label{inq}
	\|w_k\|^{\frac{N}{N-1}} < \frac{1}{2^{\frac{1}{N-1}}}\frac{2N-\mu}{ 2N}\al_N(1-l), \text{\;\; for all}\;\; k\geq k_0. 
	\end{align} 
	% So, we can choose some $q>1$, very close to $1$, such that
	%\begin{align}\label{inq-1}
	%q\|w_k\|^{\frac{N}{N-1}} \leq \frac{2N-\mu}{(h(1))^{2N}2N}\al_N-l_*
	%\end{align}  for some $0<l_*<l.$
	Now we show that
	\begin{align}\label{aaa}
	\int_{\Om}\left( \int_\Om\frac{F(y,h(w_k))}{|x-y|^{\mu}}dy\right)f(x,h(w_k))h'(w_k)w_k~dx \to 0 \;\text{as}\; k \to \infty.
	\end{align} 
	Now  using $(f_3)$, Proposition \ref{HLS}, \eqref{f},    Lemma \ref{L1}-$(h_4)$-$(h_6
	)$, H\"older's inequality, we deduce
	\begin{align}\label{wk-soln}
	&\int_{\Om}\left( \int_\Om\frac{F(y,h(w_k))}{|x-y|^{\mu}}dy\right)f(x,h(w_k))h'(w_k)w_k dx\notag\\
	&\leq\frac1\tau\int_{\Om}\left( \int_\Om\frac{f(y,h(w_k)) h(w_k)}{|x-y|^{\mu}}dy\right)f(x,h(w_k))h(w_k) dx\notag \\
	&\leq \frac 1\tau  C(N,\mu) \left(\int_{\Om}|f(x,h(w_k))h(w_k)|^{\frac{2N}{2N-\mu}}~dx\right)^{\frac{2N-\mu}{N}}\notag\\&\leq \tilde C_1(N,\mu,\e) \left(\int_{ \Om}\left(|h(w_k)|^N +  |h(w_k)|^r \exp((1+\e)|h(w_k)|^{\frac{2N}{N-1}})\right)^{\frac{2N}{2N-\mu}}~dx\right)^{\frac{2N-\mu}{N}}\notag \\
	&  \leq \tilde C_2(N,\mu,\e)\left[ \|w_k\|_{L^{\frac{2N^2}{2N-\mu}}(\Om)}^{2N}+ \|w_k\|_{L^{\frac{2Nrq'}{2N-\mu}}(\Om)}^{2r} \left(\int_{ \Om}\exp\left(q(1+\epsilon) \frac{2N}{2N-\mu}|h(w_k)|^{\frac{2N}{N-1}}\right)~dx\right)^{\frac{2N-\mu}{Nq}}\right] \notag \\
	&  \leq \tilde C_3(N,\mu,\e)\bigg[ \|w_k\|_{L^{\frac{2N^2}{2N-\mu}}(\Om)}^{2N}\notag\\&\qquad\qquad\qquad\qquad+ \|w_k\|_{L^{\frac{2Nrq'}{2N-\mu}}(\Om)}^{2r} \left(\int_{ \Om}\exp\left(q(1+\epsilon) \frac{2N}{2N-\mu}2^{\frac{1}{N-1}}\|w_k\|^{\frac{N}{N-1}}\left(\frac{|w_k|}{\|w_k\|}\right)^{\frac{N}{N-1}}\right)~dx\right)^{\frac{2N-\mu}{Nq}}\bigg],
	\end{align} where $q'=\frac{q}{q-1}$ is the conjugate of the exponent $q>1$ and $\tilde C_i's,\; i=1,2,3$ are some positive constants.
  Now we can choose $q>1$ very close to $1$ and $\e>0$  sufficiently small such that, for sufficiently large $k\in\mathbb N$, from \eqref{inq}, it follows that  $q(1+\epsilon) \frac{2N}{2N-\mu}2^{\frac{1}{N-1}}\|w_k\|^{\frac{N}{N-1}}< q(1+\e)(1-l)\al_N<\al_N$.
	% where we used the fact that $h(1)\leq1.$
	Therefore,  in the light of  Theorem \ref{TM-ineq}, for sufficiently large  $k\in\mathbb N$,  it follows that
	$$\int_\Om\exp\left(q(1+\epsilon) \frac{2N}{2N-\mu}2^{\frac{1}{N-1}}\|w_k\|^{\frac{N}{N-1}}\left(\frac{|w_k|}{\|w_k\|}\right)^{\frac{N}{N-1}}\right)<C.$$
	Using this combining with the fact  that  $w_k\to0$ strongly in $L^q(\Om)$ for $1\leq q<\infty$, as $k\to\infty,$ from \eqref{wk-soln},  we finally obtain \eqref{aaa}.
%	\[\int_{\Om}\left( \int_\Om\frac{F(y,h(w_k))}{|x-y|^{\mu}}dy\right)f(x,h(w_k))h'(w_k)w_k~dx \to 0 \;\text{as}\; k \to \infty.\]
	Since, $\{w_k\}$ is a Palais-Smale sequence for $J$, we have $\displaystyle\lim_{k\to \infty}\langle J'(w_k),w_k \rangle=0$ which together with \eqref{aaa} gives that $\ds\lim_{k\to \infty}\|w_k\|^N=0$.  Thus, from Lemma \ref{PS-ws}, it follows that $\displaystyle\lim_{k \to \infty}J(w_k)=0 =\beta^*$ which is a contradiction to the fact that $\beta^*>0$. Hence, $w \not \equiv 0$.\\
	Next, we prove that $w>0$ in $\Om$. Since, $w$ is a weak solution to \eqref{pp},
%	 Now  Lemma \ref{lem712} yields that $\{w_k\}$ is bounded.
%%	 Therefore, there exists a constant $a^*>0$ such that, up to a subsequence, $\|w_k\| \to a^*$ as $k \to \infty$.
%	  From the fact that $J^\prime(w_k) \to 0$ strongly in $\left(W_0^{1,N}(\Om)\right)^*$
%%	, it follows that, up to a subsequence, $|\nabla w_k|^{N-2}\nabla w_k \rightharpoonup |\nabla w|^{N-2}\nabla w$ weakly in $(L^{\frac{N}{N-1}}(\Om))^N$. 
%and using Lemma~\ref{kc-ws} and  Lemma~\ref{wk-sol}, when $k \to \infty$, we obtain
%	$$\int_\Om \left(\int_\Om \frac{F(y,h(w_k))}{|x-y|^{\mu}}dy\right)f(x,h(w_k)) h'(w_k)\varphi ~dx \to \int_\Om \left(\int_\Om \frac{F(y,h(w))}{|x-y|^{\mu}}dy\right)f(x,h(w))h'(w)\varphi~dx$$
%	 and
	\begin{align*}
	\int_\Om |\nabla w|^{N-2}\nabla w\nabla \varphi~dx = \int_\Om \left(\int_\Om \frac{F(y,h(w))}{|x-y|^{\mu}}dy\right)f(x,h(w))h'(w)\varphi~dx,
	\end{align*}
	for all $\varphi \in W^{1,N}_0(\Om)$. In particular, taking $\varphi = w^-$ in the above equation, we obtain $\|w^-\|=0$ which implies that $w^-=0$ a.e. in $\Om$. Therefore, $w\geq 0$ a.e. in $\Om$.\\
Now from Theorem \ref{TM-ineq}, we have $f(\cdot,h(w))\in L^q(\Om)$, for $1\leq q<\infty$.
%	Also, similarly as in \eqref{wk-sol7}, we can get $\int_\Om \frac{F(y,h(w))}{|x-y|^{\mu}}~dy \in L^\infty(\Om)$.
Also by \eqref{KC-new1}, $F(\cdot,h(w)) \in L^q(\Om),$ for any $ q \in [1,\infty)$. Since $\mu \in (0,N)$ and $\Om$ is bounded,  it follows that $y\to |x-y|^{-\mu} \in L^{q}(\Om)$ for all $q\in (1, \frac{N}{\mu})$ uniformly in $x\in \Omega$. Thus, from H\"older's inequality, we can deduce
\begin{align}\label{wk-sol7}
\int_\Om \frac{F(y,h(w))}{|x-y|^{\mu}}dy \in L^\infty(\Om).
\end{align}
	 Therefore, recalling $h'(s)\leq1$, we get $$\left(\int_\Om \frac{F(y,h(w))}{|x-y|^{\mu}}~dy \right)f(x,h(w)) h'(w)\in L^q(\Om),$$ for $1 \leq q <\infty$. Now by employing the regularity results for the elliptic equations, we infer that $w \in L^\infty(\Om)$ and $w \in C^{1,\gamma}(\overline{\Om})$ for some $\gamma \in (0,1)$.  Finally, from the strong maximum principle, we draw the conclusion that $w>0$ in $\Om$. This completes the proof of Theorem \ref{T1}.

	\section{Appendix A}
\noi In this section,  we give the proof of the following almost everywhere convergence of gradients of Palais-Smale sequences using the concentration compactness arguments. 
	\begin{lemma}\label{wk-sol} Suppose the assumptions $(f_1)$-$(f_6)$   hold and  $h$ is defined as in \eqref{g}. 	 Let $\{w_k\}\subset W^{1,N}_0(\Om)$ be a Palais-Smale sequence for $ J$. Then $\nabla w_k \rightarrow \na w$ a.e. in $\Omega$. Moreover, we have, as $k\to\infty$
		%	If $\{w_k\}$ denotes a Palais-Smale sequence then up to a subsequence, there exists $w \in W^{1,N}_0(\Om)$ such that
		%\begin{equation}\label{wk-sol1}
		%(|x|^{-\mu}\ast F(u_k))f(u_k) \to (|x|^{-\mu}\ast F(u))f(u) \; \text{in}\; L^1(\Om)
		%\end{equation}
		\begin{equation}\label{wk-sol2}
		|\nabla w_k|^{N-2}\nabla w_k \rightharpoonup |\nabla w|^{N-2}\nabla w\; \text{weakly in}\; (L^{\frac{N}{N-1}}(\Om))^N.
		\end{equation}
	\end{lemma}
	\begin{proof}
		%	From Lemma \ref{lem7.1}, we have that the sequence $\{w_k\}$ is bounded in $W^{1,N}_0(\Om)$. So, up to a subsequence, still denoted by $\{w_k\}$, there exists $w\in W^{1,N}_0(\Omega)$ such that $w_k \rightharpoonup w$ weakly in $W_0^{1,N}(\Om)$ and strongly in $L^q(\Om)$ for any $q \in [1,\infty)$ as $k \to \infty$. Hence $w_k(x) \to w(x)$ point-wise a.e. for $ x \in \Om$.
		Since Lemma \ref{lem712} yields that $\{w_k\}$  is bounded in $W^{1,N}_0(\Om)$,  there exists  $w\in W^{1,N}_0(\Om)$ such that, in the sense of subsequence, we have 
		$w_k \rightharpoonup w$ weakly in $ W^{1,N}_0(\Om)$,
		$w_k \to w$  strongly in $L^q(\Om),\, q \in [1,\infty)$,
		$w_k(x) \to w(x)$ point-wise a.e. in $ \Om,$
		as $k \to \infty$.
		From the properties of the sequence $\{w_k\}$, it is evident that the sequence $\{|\nabla w_k|^{N-2}\nabla w_k\}$ must be bounded in $(L^{\frac{N}{N-1}}(\Om))^N$, which implies that 
		there exists $u \in (L^{\frac{N}{N-1}}(\Om))^N$ such that,
		\begin{align*}|\nabla w_k|^{N-2}\nabla w_k \rightharpoonup u \; \text{weakly in}\;  (L^{\frac{N}{N-1}}(\Om))^N \; \text{as} \; k \to \infty.\end{align*}
		Also we have, $\{|\nabla w_k|^N\}$ is bounded in $L^1(\Om)$,  which yields that there exists a non-negative radon measure $\sigma$ such that, up to a subsequence, we have
		\begin{align*}|\nabla w_k|^N \to \sigma \; \text{in}\; (C(\overline{\Om}))^*\; \text{as}\; k \to \infty.\end{align*}
		Our aim is to show $u = |\nabla w|^{N-2}\nabla w $.
		For that, first we take $\nu>0$ and set $X_\nu := \{x \in \overline \Om:\; \sigma(B_l(x)\cap \overline \Om)\geq \nu, \;\text{for all}\; l>0\}$.\\
		{\bf{Claim 1:}}  $X_\nu$ is a finite set. \\
		Indeed, if not, then there exists a sequence of distinct points $\{z_k\}$ in $X_\nu$ such that, $\sigma(B_l(z_k)\cap \overline \Om)\geq \nu$ for all $l>0$ and for all $k\in\mathbb N$. This gives that $\sigma(\{z_k\}) \geq \nu$ for all $k$. Therefore, $\sigma(X_\nu)= +\infty$. But this is a contradiction to the fact that
		\[\sigma(X_\nu) = \lim_{k \to \infty} \int_{X_\nu} |\nabla w_k|^N ~dx \leq C.\] Hence, the claim holds.
		Thus, we can take $X_\nu= \{z_1,z_2,\ldots, z_n\}$.\\
		{\bf	{Claim 2:} } We can choose  $\nu>0,$ such that $\nu^{\frac{1}{N-1}} <\frac{1}{2^{N-1}}  \frac{2N-\mu}{2N} \alpha_N$ and we have
		\begin{align}\label{wk-sol6}
		&\lim_{k \to \infty} \int_S\left( \int_\Om \frac{F(y, h (w_k))}{|x-y|^{\mu}}dy\right) f(x,h(w_k))h'(w_k)w_k~dx\notag\\ &\qquad=  \int_S \left( \int_\Om \frac{F(y,h(w))}{|x-y|^{\mu}}dy\right) f(x,h(w)) h'(w)w~dx, \end{align} where  $S$ is any compact subset of $\oline \Om \setminus X_\nu$.\\
		Let $z_0 \in S$ and $l_0>0$ be such that $\sigma(B_{l_0}(z_0) \cap \oline \Om) < \nu$ that is $z_0 \notin X_\sigma$. Also, we consider  $\phi \in C_c^\infty(\Om)$ satisfying $0\leq \phi(x)\leq 1$ for $x \in \Om$, $\phi \equiv 1$ in $B_{\frac{l_0}{2}}(z_0)\cap \oline \Om$ and $\phi \equiv 0$ in $\oline \Om \setminus (B_{l_0}(z_0)\cap \oline \Om)$. Then
		\[\lim_{k \to \infty} \int_{B_{\frac{l_0}{2}}(z_0) \cap \oline \Om}|\nabla w_k|^N dx\leq \lim_{k \to \infty} \int_{B_{l_0}(z_0) \cap \oline \Om}|\nabla w_k|^N\phi dx \leq \sigma(B_{l_0}(z_0) \cap \oline \Om) < \nu. \]
		Hence, for sufficiently large $k \in \mb N$ and  sufficiently $\e>0$ small, it follows that
		\begin{equation}\label{wk-sol3}
		\int_{B_{\frac{l_0}{2}}(z_0) \cap \oline \Om}|\nabla w_k|^N \leq \nu(1-\e).
		\end{equation}
		Now we estimate the following using $(f_3)$, \eqref{wk-sol3} and  Lemma \ref{L1}-$(h_6
		)$:
		\begin{equation*}
		\begin{split}
		%&\int_{B_{\frac{r_0}{2}}(x_0) \cap \oline \Om} \left(\int_\Om \frac{F(x,u_k)}{|x-y|^\mu}dy\right)|f(x,u_k)|^q~dx
		& \int_{B_{\frac{l_0}{2}}(z_0) \cap \oline \Om}|f(x,h(w_k))|^q~dx=  \int_{B_{\frac{l_0}{2}}(z_0) \cap \oline \Om}|g(x,h(w_k))|^q\exp\left(q|h(w_k)|^{\frac{2N}{N-1}}\right)~dx \\
		&  \leq C \int_{B_{\frac{l_0}{2}}(z_0) \cap \oline \Om}\exp\left((1+\epsilon)2^{\frac{1}{N-1}}q|h(w_k)|^{\frac{2N}{N-1}}\right)~dx\\
		&  \leq C\int_{B_{\frac{l_0}{2}}(z_0) \cap \oline \Om}\exp\left((1+\epsilon)2^{\frac{1}{N-1}}q|w_k|^{\frac{N}{N-1}}\right)~dx\\
		& \leq C \int_{B_{\frac{l_0}{2}}(z_0) \cap \oline \Om}\exp\bigg((1+\epsilon)q 2^{\frac{1}{N-1}}\nu^{\frac{1}{N-1}}(1-\e)^{\frac{1}{N-1}}\bigg(\frac{|w_k|^N}{\int_{B_{\frac{l_0}{2}}(z_0) \cap \oline \Om}|\nabla w_k|^N dx}\bigg)^{\frac{1}{N-1}}\bigg)~dx.
		\end{split}
		\end{equation*}
		Thus, we can choose $q>1$ and   $\e>0$ such that $(1+\e) 2^{\frac{1}{N-1}}q \nu^{\frac{1}{N-1}} < \alpha_N$ and then using Theorem \ref{TM-ineq} in the last relation, we obtain
		\begin{align}\label{aa}
		& \int_{B_{\frac{l_0}{2}}(z_0) \cap \oline \Om}|f(x,h(w_k))|^q~dx\leq C
		\end{align} for some constant $C>0$, independent of $k.$ Next, we consider
		\begin{align}\label{a2}
		&\int_{B_{\frac{l_0}{2}}(z_0) \cap \oline \Om} \left|\left( \int_\Om \frac{F(y,h(w_k))}{|x-y|^{\mu}}dy\right) f(x,h(w_k))h'(w_k) w_k- \left( \int_\Om \frac{F(y,h(w))}{|x-y|^{\mu}}dy\right) f(x,h(w))h'(w) w \right|~dx\notag\\
		& \leq \int_{B_{\frac{l_0}{2}}(z_0) \cap \oline \Om} \left|\left( \int_\Om \frac{F(y,h(w))}{|x-y|^{\mu}}dy\right) (f(x,h(w_k))h'(w_k) w_k-f(x,h(w))h'(w))w\right|~dx\notag\\
		& \quad + \int_{B_{\frac{l_0}{2}}(z_0) \cap \oline \Om} \left|\left( \int_\Om \frac{F(y,h(w_k))-F(y,h(w))}{|x-y|^{\mu}}dy\right) f(x,h(w_k))h'(w_k) w_k\right|~dx\notag\\
		&:= J_1 + J_2.
		\end{align}
%		From \eqref{KC-new1}, we know that $F(x,h(w)) \in L^r(\Om)$ for any $ r \in [1,\infty)$. Since $\mu \in (0,N)$ and $\Om$ is bounded then using the fact that $y\to |x-y|^{-\mu} \in L^{r}(\Om)$ for all $r\in (1, \frac{N}{\mu})$ uniformly in $x\in \Omega$ and applying H\"older's inequality, we can deduce
%		\begin{equation}\label{wk-sol7}
%		\int_\Om \frac{F(y,h(w))}{|x-y|^{\mu}}dy \in L^\infty(\Om).
%		\end{equation}
		From the asymptotic growth assumptions on  $f$, we obtain
		\begin{equation}\label{wk-sol8}
		\lim_{s \to \infty} \frac{f(x,s
			)s}{(f(x,s))^r} = 0\; \text{uniformly in }x \in \Omega,\; \text{for all}\; r>1.
		\end{equation}
		Using \eqref{wk-sol7}, we get
		\[
		J_1 \leq  C\int_{B_{\frac{l_0}{2}}(z_0) \cap \oline \Om}  |f(x,h(w_k))h'(w_k)w_k-f(x,h(w))h'(w)w|~dx,\]
		where $C>0$ is a constant, independent on $k$. Moreover, \eqref{aa}  and \eqref{wk-sol8} together with Lemma \ref{L1}-$(h_4)$ imply that   $\{f(x,h(w_k))h'(w_k) w_k\}$ is an equi-integrable family of functions over ${B_{\frac{l_0}{2}}(z_0) \cap \oline \Om} $. Also, by the continuity of the functions $f $ and $h$, we have   $f(x,h(w_k))h'(w_k)w_k \to f(x,h(w))h'(w)w$  a.e. in $\Om$ as $k \to \infty.$ Hence, by applying Vitali's convergence theorem, we obtain $$J_1 \to 0 \text {\;\;\;\; as \;\;} k\to \infty.$$ Next, we show that $J_2 \to 0$ as $k \to \infty$.\\
	For that, first we recall the semigroup property of the Riesz potential and get the following estimate:
		\begin{align}\label{r}
		J_2=&\int_{\Omega} \left(\int_{\Omega}\frac{F(y,h(w_k))-F(y,h(w))}{|x-y|^{\mu}} dy \right) \chi_{B_{\frac{l_0}{2}}(z_0) \cap \overline{\Omega}}(x) f(x,h(w_k))h'(w_k) w_k ~dx \notag \\
		&\leq  C\left(\int_{\Omega} \left( \int_{\Omega}\frac{|F(y,h(w_k))- F(y,h(w))| dy }{|x-y|^{\mu}}\right) |F(x,h(w_k))- F(x,h(w))| ~dx \right)^{\frac12}\notag\\
		&\quad \times \left(\int_{\Omega} \left(\int_{\Omega} \chi_{B_{\frac{l_0}{2}}(z_0) \cap \overline{\Omega}}(y) \frac{f(y,h(w_k))h'(w_k) w_k}{|x-y|^{\mu}} dy \right) \chi_{B_{\frac{l_0}{2}} (z_0)\cap \overline{\Omega}}(x) f(x,h(w_k)) h'(w_k) w_k ~dx \right)^{\frac12},
		\end{align} where  $C>0$  is some constant, independent of $k$.
		Combining \eqref{aa}, \eqref{wk-sol8}, Lemma \ref{L1}-$(h_4)$,  Proposition \ref{HLS} and  $\nu^{\frac{1}{N-1}} < \frac{1}{2^{N-1}}  \frac{2N-\mu}{2N} \alpha_N$, it follows that
		\begin{align*}
		&\left(\int_{\Omega} \left(\int_{\Omega} \chi_{B_{\frac{l_0}{2}} (z_0)\cap \overline{\Omega}}(y)\frac {f(y,h(w_k))h'(w_k) w_k}{|x-y|^\mu} dy \right) \chi_{B_{\frac{l_0}{2}}(z_0) \cap \overline{\Omega}}(x) f(x,h(w_k))h'(w_k) w_k ~dx \right)^{\frac12}\\
			&\leq \left(\int_{\Omega} \left(\int_{\Omega} \chi_{B_{\frac{l_0}{2}} (z_0)\cap \overline{\Omega}}(y)\frac {f(y,h(w_k))h(w_k)}{|x-y|^\mu} dy \right) \chi_{B_{\frac{l_0}{2}}(z_0) \cap \overline{\Omega}}(x) f(x,h(w_k))h(w_k) ~dx \right)^{\frac12}\\& \leq \|\chi_{B_{\frac{l_0}{2}} (z_0)\cap \overline{\Omega}}f(\cdot,h(w_k))h(w_k) \|_{L^{\frac{2N}{2N-\mu}}(\Omega)} \leq C.
		\end{align*}
		This together with   \eqref{r} and Lemma \ref{PS-ws} gives that $$J_2 \to 0 \text{\;\;\;\; as \;\; } k \to \infty.$$
		%Lastly we estimate the following while choosing $K > \max\{T, \}$, using (h5)
		%\begin{equation}
		%\begin{split}
		%&\int_{\{|u|\leq M\}}\left(\int_{\{|u_k|\geq K\}}\frac{F(y,u_k)}{|x-y|^\mu}dy\right)F(x,u)~dx \\
		%& = \int_{\{|u|\leq M\}}\left(\int_{\{|u_k|\leq K\}\cap \{F(u_k)\leq 1\}}\frac{F(y,u_k)}{|x-y|^\mu}dy\right)F(x,u)~dx + \int_{\{|u|\leq M\}}\left(\int_{\{|u_k|\leq K\}\cap\{F(u_k)>1\}}\frac{F(y,u_k)}{|x-y|^\mu}dy\right)F(x,u)~dx\\
		%&\leq \int_{\{|u|\leq M\}}\left(\int_{\{|u_k|\leq K\}}\frac{1}{|x-y|^\mu}dy\right)F(x,u)~dx + \frac{T_0}{K^{\gamma_0+1}}\int_{\{|u|\leq M\}}\left(\int_{\{|u_k|\leq K\}\cap \{F(u_k)\leq 1\}}\frac{u_k f(y,u_k)}{|x-y|^\mu}dy\right)F(x,u)~dx
		%\end{split}
		%\end{equation}
		Hence, from \eqref{a2}, we can infer that 
		\begin{align*}&\lim_{k \to \infty}\int_{B_{\frac{l_0}{2}}(z_0) \cap \oline \Om} \Bigg|\left( \int_\Om \frac{F(y,h(w_k))}{|x-y|^{\mu}}dy\right) f(x,h(w_k))h'(w_k)w_k\\&\qquad\qquad\qquad\qquad\qquad- \left( \int_\Om \frac{F(y,h(w))}{|x-y|^{\mu}}dy\right) f(x,h(w))h'(w)w \Bigg|~dx=0.\end{align*}
		Since $S$ is compact, we can repeat this procedure over a finite covering of balls and hence,  \eqref{wk-sol6} is achieved. Finally, the proof of \eqref{wk-sol2} can be concluded by the similar standard arguments as in the proof of Lemma $4$ in \cite{M-do1}. 
\end{proof} 
\noindent{\bf Acknowledgements:} 
The second author would like to thank the Science and Engineering Research Board, Department of Science and Technology,
Government of India for the financial support under the grant MTR/2018/001267.                                                   

\end{document}